\documentclass[11 pt]{amsart}

\usepackage{amsmath,amsthm,amssymb,amscd}

\newcommand{\A}{\mathcal{A}}
\newcommand{\B}{\mathcal{B}}
\newcommand{\C}{\mathcal{C}}

\renewcommand{\P}{\mathcal{P}}
\newcommand{\K}{\mathbb{K}}
\newcommand{\Q}{\mathbb{Q}}

\newcommand{\R}{\mathbb{R}}
\newcommand{\dom}{\mathrm{dom}}

\newcommand{\N}{\mathbb{N}}

\newcommand{\iso}{\mathrm{Iso}}
\newcommand{\aut}{\mathrm{Aut}}
\newcommand{\U}{\mathbb{U}}

\newcommand{\ball}{\mathrm{ball}}
\newcommand{\sphere}{\mathrm{S}}

\newcommand{\diam}{\mathrm{diam}}
\newcommand{\id}{\mathrm{id}}

\newcommand{\elomega}{\mathcal{L}_{\omega_1\omega}}
\newcommand{\mtp}{\mathrm{qftp}}
\newcommand{\cmtp}{\mathrm{qftp}}
\newcommand{\rng}{\mathrm{rng}}
\newcommand{\MALG}{\mathrm{MALG}}
\newcommand{\malg}{\mathrm{MALG}}
\newcommand{\lspan}{\mathrm{span}}

\newcommand{\gr}{\mathrm{Gr}}
\newcommand{\at}{\mathrm{atom}}

\def\Ind#1#2{#1\setbox0=\hbox{$#1x$}\kern\wd0\hbox to 0pt{\hss$#1\mid$\hss}
\lower.9\ht0\hbox to 0pt{\hss$#1\smile$\hss}\kern\wd0}
\def\Notind#1#2{#1\setbox0=\hbox{$#1x$}\kern\wd0\hbox to
0pt{\mathchardef\nn="0236\hss$#1\nn$\kern1.4\wd0\hss}\hbox
to 0pt{\hss$#1\mid$\hss}\lower.9\ht0
\hbox to 0pt{\hss$#1\smile$\hss}\kern\wd0}
\def\ind{\mathop{\mathpalette\Ind{}}}

\newtheorem{theorem}{Theorem}[section]
\newtheorem{lemma}[theorem]{Lemma}
\newtheorem{corollary}[theorem]{Corollary}
\newtheorem{proposition}[theorem]{Proposition}
\newtheorem{claim}[theorem]{Claim}

\theoremstyle{definition}
\newtheorem{definition}[theorem]{Definition}

\newtheorem{question}[theorem]{Question}

\title{Automatic continuity for isometry groups}

\author{Marcin Sabok}\thanks{The author acknowledges partial
  support from the following sources: the NCN (the Polish
  National Science Centre) grant no. 2012/05/D/ST1/03206 and
  the MNiSW (the Polish Ministry of Science and Higher
  Education) grant no. 0435/IP3/2013/72.}

\address{Instytut Matematyczny Polskiej Akademii Nauk,
  ul. \'Sniadeckich 8, 00-956 Warszawa, Poland}
\address{Instytut Matematyczny Uniwersytetu Wroc\l awskiego,
  pl. Grunwaldzki 2/4, 50--384 Wroc\l aw, Poland}

\email{M.Sabok@impan.pl}

\date{}

\subjclass[2010]{03E15, 54H11, 37B05}

\begin{document}

\begin{abstract}
  We present a general framework for automatic continuity
  results for groups of isometries of metric spaces. In
  particular, we prove automatic continuity property for the
  groups of isometries of the Urysohn space and the Urysohn
  sphere, i.e. that any homomorphism from either of these
  groups into a separable group is continuous. This answers
  a question of Melleray. As a consequence, we get that the
  group of isometries of the Urysohn space has unique Polish
  group topology and the group of isometries of the Urysohn
  sphere has unique separable group topology. Moreover, as
  an application of our framework we obtain new proofs of
  the automatic continuity property for the group
  $\aut([0,1],\lambda)$, due to Ben Yaacov, Berenstein and
  Melleray and for the unitary group of the
  infinite-dimensional separable Hilbert space, due to
  Tsankov. The results and proofs are stated in the language
  of model theory for metric structures.
\end{abstract}

\maketitle

\section{Introduction}

It is well known that every group is isomorphic to the group
of isometries of a metric space (or even of a
graph). Moreover, if $G$ is the group of isometries of a
metric space $X$, then $G$ carries the topology of pointwise
convergence on $X$. If the space $X$ is separable and its
metric is complete, then $G$ is separable and completely
metrizable (i.e. \textit{Polish}). In fact, the coverse is
also true and any Polish group is isomorphic to the group of
isometries of a separable complete metric space
\cite[Theorem 3.1]{gao.kechris}. Note that if a group is
isomorphic to the group of isometries of a space $X$, then
the structure of this group is completely determined by the
metric properties of $X$. In this paper, we exploit this
observation to study the structure of various groups of isometries.

Automatic continuity is a phenomenon that connects the
algebraic and topological structures and typically says that
any map which preserves an algebraic structure must
automatically be continuous. One of the first instances of
this phenomenon appears in C*-algebras and Banach algebras,
where it is known that any homomorphism from an abelian
Banach algebra into $\mathbb{C}$ is continuous. More
nontrivial results concern continuity of derivations on
C*-algebras. Sakai \cite{sakai} (proving a conjecture of
Kaplansky \cite{kaplansky}) showed that any derivation on a
C*-algebra is norm-continuous. This was generalized first by
Kadison \cite{kadison} who improved it to the continuity in
the ultraweak topology and then by Ringrose \cite{ringrose}
for derivations of C*-algebras into Banach modules. Johnson
and Sinclair \cite{johnson.sinclair} on the other hand,
showed automatic continuity for derivations on semi-simple
Banach algebras. A detailed account on this subject can be
found in \cite{dales.survey,dales.book}.

In the context of groups and their homomorphisms, one of the
first automatic continuity results has been proved by Dudley
\cite{dudley}, who showed that any homomorphism from a
complete metric or a locally compact group into a normed
(e.g. free) group is continuous (see also \cite{slutsky} for
a recent generalization to homomorphisms into free
products). A general form of automatic continuity phenomenon
for groups has appeared in the work of Kechris and Rosendal
\cite{kechris.rosendal}, with connection to the results of
Hodges, Hodkinson, Lascar and Shelah \cite{hhls}.

A topological group $G$ has the \textit{automatic continuity
  property} if for every separable topological group $H$,
any group homomorphism from $G$ to $H$ is continuous. Recall
\cite[Theorem 9.10]{kechris} that any measurable
homomorphism from a Polish group to a separable group must
be continuous and the existence of non-measurable
homomorphisms on groups such as $(\mathbb{R},+)$ can be
derived from the axiom of choice. So, similarly as
amenability, automatic continuity property for a given group
can be interpreted in terms of nonexistence (on this group)
of pathological phenomena that can follow from the axiom of
choice.

Kechris and Rosendal \cite{kechris.rosendal}
showed that automatic continuity is a consequence of the
existence of comeager orbits in the diagonal conjugacy
actions of $G$ on $G^n$ for each $n\in\N$
(i.e. \textit{ample generics}, cf \cite[Section
1.6]{kechris.rosendal}) and discovered a connection between
ample generics for the automorphism groups of homogeneous
structures and the Fra\"iss\'e theory. In consequence, many
automorphism groups of homogeneous structures turned out to
have the automatic continuity property. However, automatic
continuity can hold also for groups which do not have ample
generics (and even have meager conjugacy classes). Rosendal
and Solecki \cite{rosendal.solecki} proved that automatic
continuity property holds for the groups of homeomorphisms
of the Cantor space and of the real line, and for the
automorphism group of $(\mathbb{Q},<)$. Rosendal
\cite{rosendal.manifolds} showed automatic continuity for
the groups of homeomorphisms of compact 2-manifolds. A
survey on recent results in this area can be found in
\cite{rosendal.survey}.

The Urysohn space $\U$ is the separable complete metric
space which is \textit{homogeneous} (i.e. any finite partial
isometry of $\U$ extends to an isometry of $\U$) and such
that any finite metric space embeds into $\U$
isometrically. It is known that these properties of $\U$
determine it uniquely up to isometry and that any separable
metric space embeds into $\U$ \cite[Theorem
5.1.29]{pestov}. The analogue of the Urysohn space of
diameter 1 also exists and is called the \textit{Urysohn
  sphere} (or the \textit{bounded Urysohn space of diameter
  1}) and denoted by $\U_1$ (see \cite[Remark
5.1.31]{pestov}).  The group of isometries of $\U$ is
universal among Polish groups, i.e. any Polish group is its
closed subgroup \cite[Theorem 2.5.2]{gao}. The Urysohn
space, as well as its group of isometries, have received a
considerable amount of attention recently. Tent and Ziegler
\cite{tent.ziegler} showed that the quotient of the group of
isometries of $\U$ modulo the normal subgroup of bounded
isometries is a simple group and recently
\cite{tent.ziegler.new} proved that the group of isometries
of the bounded Urysohn space is simple. For more on the
structure of the Urysohn space and its group of isometries,
see the recent monographs \cite{pestov,gao.kechris,gao} on
this subject.

Kechris and Rosendal showed that the group of automorphisms
of the \textit{rational Urysohn space} (which is the Fra\"\i
ss\'e limit of the class of finite metric spaces with
rational distances) has ample generics and deduced from it
that the group $\iso(\U)$ has a dense conjugacy class
\cite[Theorem 2.2]{kr}. The question whether the group of
isometries of the Urysohn space has the automatic continuity
property has been asked by Melleray (cf. \cite[Section
6.1]{bybm}). One of the main applications of the results of
this paper is the following.

\begin{theorem}\label{urysohn}
  The groups of isometries of the Urysohn space and the
  Urysohn sphere have the automatic continuity property.
\end{theorem}

Theorem \ref{urysohn} has some immediate consequences on the
topological structure of the above groups, in spirit of the
results of Kallman \cite{kallman.malg,kallman.homeo} and
Atim and Kallman \cite{atim.kallman}. The first one is an
abstract consequence of automatic continuity

\begin{corollary}
  The group $\iso(\U)$ has unique Polish group topology.
\end{corollary}

Recall that a group is \textit{minimal} if it does not admit
any strictly corser (Hausdorff) group topology (in this
paper we consider only Hausdorff topologies on groups).  The
second corollary follows from minimality of the group of
isometries of the Urysohn sphere, proved by Uspenskij
\cite{uspenskij}

\begin{corollary}
  The group $\iso(\U_1)$ has unique separable group
  topology.
\end{corollary}

Theorem \ref{urysohn} will follow from the following
abstract result, which isolates metric (or model-theoretic)
properties of a metric structure that imply that the group
of automorphisms (with the pointwise convergence topology)
of the structure has the automatic continuity property. The
definitions of a metric structure and the notions appearing
in the statement of the theorem are given in Sections
\ref{sec:first.order} and \ref{sec:metric.structures}.

\begin{theorem}\label{abstract}
  Suppose $M$ is a homogeneous complete metric structure
  that has locally finite automorphisms, the extension
  property and admits weakly isolated sequences. Then the
  group $\aut(M)$ has the automatic continuity property.
\end{theorem}

Theorem \ref{abstract} can be also applied to give a unified
treatment of previously known automatic continuity results
for automorphism groups of some metric structures. It is
worth mentioning that up to now, these results have been
proved with different methods, varying from case to case. In
this paper, we apply Theorem \ref{abstract} to show the
automatic continuity property for the group
$\aut([0,1],\lambda)$ (the group of measure-preserving
automorphism of the unit interval) and the group $U(\ell_2)$
(unitary operators of the infinite-dimensional separable
Hilbert space).

Automatic continuity for the group $\aut([0,1],\lambda)$ has
been proved in a series of two papers
\cite{kittrell.tsankov,bybm}. Kittrell and Tsankov
\cite{kittrell.tsankov} showed first that any homomorphism
from $\aut([0,1],\lambda)$ to a separable group must be
continuous in the strong topology of $\aut([0,1],\lambda)$
(see \cite[Section 1]{kechris.global}) and later, Ben
Yaacov, Berenstein and Melleray \cite{bybm} proved a general
result which implies that any homomorphism which is
continuous in the strong topology on $\aut([0,1],\lambda)$
must be continuous in the weak topology on
$\aut([0,1],\lambda)$ (this approach has been recently
simplified by Malicki \cite{malicki}). The group
$\aut([0,1],\lambda)$ (with the weak topology) is isomorphic
to the group of automorphism of the measure algebra and
applying Theorem \ref{abstract} to the measure algebra we
get a new proof of automatic continuity.

\begin{corollary}[Ben Yaacov, Berenstein, Melleray]\label{malg}
  The group $\aut([0,1],\lambda)$ has the automatic
  continuity property.  
\end{corollary}

After the work of Ben Yaacov, Berenstein and Melleray
\cite{bybm}, Tsankov \cite{tsankov} further showed the
automatic continuity property for the infinite-dimensional
unitary group. Given that the group $U(\ell_2)$ is the
automorphism group of the Hilbert space $\ell_2$ (or the
isometry group of the sphere in $\ell_2$), we can apply
Theorem \ref{abstract} to the Hilbert space and get a new
proof of this result.

\begin{corollary}[Tsankov]\label{hilbert}
  The group $U(\ell_2)$ has the automatic continuity
  property.
\end{corollary}

Our proof of Theorem \ref{abstract} builds on the work of
Kechris and Rosendal \cite{kr} and Rosendal and Solecki
\cite{rosendal.solecki}. In particular, we use the notion of
ample generics introduced in \cite{kr} and exploit some
ideas of \cite[Section 3]{rosendal.solecki}. The
verification of conditions of Theorem \ref{abstract} in the
case of the Urysohn space uses a result of Solecki
\cite{sol.iso} that is based on earlier results of Herwig
and Lascar \cite{herwig.lascar}. These result in turn, are
connected to a theorem of Ribes and Zalesski\u\i\
\cite{ribes.zalesskii}, who showed separability of products
of finitely-generated subgroups of the free groups (this was
later generalized by Minasyan \cite{minasyan} to hyperbolic
groups). In Section \ref{sec:urysohn} we present a new proof
of Solecki's theorem \cite{sol.iso} in the style of Mackey's
construction of induced actions and based on the
separability result of Ribes and Zalesski\u\i.

This paper is organized as follows. In Sections
\ref{sec:first.order} and \ref{sec:metric.structures} we
give an overview of model-theoretic notions that appear in
the statement of Theorem \ref{abstract}. In Sections
\ref{sec:continuous.logic} and \ref{sec:ample.generics} we
prove a weak version of ample generics for the automorphism
groups of metric structures and Section
\ref{sec:automatic.continuity} includes a proof of Theorem
\ref{abstract}. Section \ref{sec:trivial} contains a further
result on triviality of homomorphisms to groups admitting
complete left-invariant metrics. In Section
\ref{sec:urysohn} we verify that the assumptions of Theorem
\ref{abstract} are satisfied by the Urysohn space, which
proves Theorem \ref{urysohn}. Sections \ref{sec:measure} and
\ref{sec:hilbert} contain discussion of the cases of the
measure algebra and the Hilbert space and proofs of
Corollaries \ref{malg} and \ref{hilbert}.

\section{First-order continuous model
  theory}\label{sec:first.order}

The techniques used in this paper are motivated by the
framework and language of model theory for metric
structures, developed recently by Ben Yaacov, Berenstein,
Henson, Usvyatsov \cite{model.theory} and others. There are,
however, some details, that will vary from the original
setting. In this paper, a \textit{metric structure} is a
tuple $(X,d_X,f_1,f_2,\ldots)$ where $X$ equipped with $d_X$
is a separable metric space and $f_1,f_2,\ldots$ are either
closed subsets of $X^n\times\R^m$ (\textit{relations}), for
some $n,m\in\N$, or continuous functions, from $X^n\times
\R^m$ to $X^k\times \R^l$ for some $n,m,k,l\in\N$ (here we
consider the discrete topology on $\R$, i.e. demand
continuity only on the arguments from $X$). Thus, a metric
structure is a two-sorted structure with the second sort
being a subset of the real line $\R$. Let us note here that
in the examples considered in this paper, the structures
will contain no relational symbols (only functions) and we
allow them in the definition only for sake of generality.

We do not require our metric structures to be complete (as
metric spaces) and we say that a metric structure $X$ is
\textit{complete} if it is complete with respect to $d_X$.

Given a metric structure $M$ we write $\aut(M)$ for the
group of automorphisms of $M$ (i.e. bijections of the first
sort of $M$ which preserve both the metric and each
$f_n$). $\aut(M)$ is always endowed with the topology of
pointwise convergence and if $M$ is complete, then $\aut(M)$
is a Polish group. Say that a metric structure $X$ is
\textit{homogeneous} if every partial isomorphism between
finitely generated substructures of $X$ extends to an
automorphism of $X$. Note that in case $X$ is a metric
space, this coicides with the usual notion of homogeneity
(sometimes also referred to as \textit{ultrahomogeneity}).

The main difference between continuous logic and our
approach lies the syntax. We will consider only a
first-order variant of the language. \textit{Terms} are
either variables, elements of a structure (of the first sort
or the second sort), or expressions of the form
$f(\tau_1,\ldots,\tau_n)$, where $\tau_1,\ldots,\tau_n$ are
terms and $f$ is a function symbol (e.g. the symbol for the
distance function) of appropriate kind (where the numbers
and sorts of the variables are correct). The
\textit{first-order formulas} in our language will be of the
form
\begin{itemize}
\item $\tau=\sigma$, or $R(\tau_1,\ldots,\tau_n)$, where
  $\tau,\sigma,\tau_1,\ldots,\tau_n$ are terms and $R$ is a
  relational symbol,
\item if $\varphi$ and $\psi$ are first-order formulas and
  $x$ is a variable of the first sort, then $\exists x
  \varphi$, $\forall x\varphi$, $\neg\varphi$,
  $\varphi\vee\psi$ are first-order formulas as well
  (quantification is only allowed over the first sort).
\end{itemize}

As usual, a \textit{first-order sentence} is a first-order
formula without free variables. The truth value of a
first-order sentence in a metric structure is defined as in
the classical setting (it is either $0$ or $1$). We will use
the symbol $\prec$ for an elementary substructure, in the
following (classical) sense: given a metric structure $X$
and its substructure $Y\subseteq X$ we write $Y\prec X$ if
for every first order sentence $\sigma$ with parameters in
$Y$, if $\sigma$ is true in $Y$, then $\sigma$ is true in
$X$.

Given a metric structure $X$ and a tuple $\bar
a=(a_1,\ldots,a_m)$ of elements of $X$, a
\textit{quantifier-free type over} $\bar a$ is a set of
quantifier-free formulas $\varphi(\bar x,\bar a)$ for a
fixed sequence of variables $\bar x=(x_1,\ldots,x_n)$ of the
first sort. A \textit{quantifier-free type} is a
quantifier-free type over the empty tuple. Quantifier-free
types will be denoted by $p(\bar x)$ (to indicate the
variables), or simply $p$. If $\bar x=(x_1,\ldots,x_n)$,
then we also say that $p(\bar x)$ is a
\textit{quantifier-free $n$-type over $\bar a$}.

We say that an $n$-tuple $\bar b=(b_1,\ldots,b_n)$ of
elements of a metric structure $X$ \textit{realizes} a given
quantifier-free $n$-type $p$ over $\bar a$ (write $\bar
b\models p$) if $X\models\varphi(\bar b,\bar a)$ for every
$\varphi(\bar x,\bar a)\in p$ (the definition of
satisfaction in a model is the natural one). Given
$Y\subseteq X$ with $\bar a\in Y^n$ and a quantifier-free
type $n$-type $p$ over $\bar a$ we say that \textit{$p$ is
  realized in $Y$} if there is $\bar b\in Y^n$ such that
$\bar b\models p$. Also, abusing notation, if $Y\subseteq
X^k$, then we say that \textit{$p$ is realized in $Y$} if if
there is $\bar b\in Y$ such that $\bar b\models p$. The
\textit{quantifier-free type of a tuple} $\bar b$ over $\bar
a$, denoted by $\mtp(\bar b\slash\bar a)$ is then the set of
all quantifier-free formulas $\varphi(\bar x,\bar a)$ such
that $X\models\varphi(\bar b,\bar a)$. If $\bar a$ is the
empty tuple, then we write $\mtp(\bar b)$ for $\mtp(\bar
b\slash\bar a)$. A quantifier-free type is
\textit{consistent} if there is a model that realizes it,
and a consistent quantifier-free type $p$ is
\textit{complete} if whenever $p\subseteq q$ and $q$ is a
consistent quantifer-free type, then $p=q$.

\begin{definition}
  Given $n\in\N$, a quantifier-free $n$-type $p$ over $\bar
  a=(a_1,\ldots,a_n)$ and $\varepsilon>0$, say that $p$ is
  an \textit{$\varepsilon$-quantifier-free $n$-type over
    $\bar a$} if $\mtp(\bar a)\subseteq p$ and
  $d(x_i,a_i)=\varepsilon_i$ belongs to $p$ for each $i\leq
  n$ and for some $0\leq \varepsilon_i<\varepsilon$.
\end{definition}

Given a metric structure $X$, $k\in\N$ and a complete
quantifier-free $k$-type $p$ write $p(X)=\{\bar a\in X^k:
\bar a\models p\}$ and note that $p(X)\subseteq X^k$ is
$G_\delta$ (closed if there are no relation symbols) in the
topology of $X^k$, so if $X$ is complete, then $p(X)$
becomes a Polish space.

Given three tuples $\bar a,\bar b$ and $\bar c$ in a metric
structure $X$, write $\bar b\equiv_{\bar a} \bar c$ if
$\mtp(\bar b\slash\bar a)=\mtp(\bar c\slash \bar a)$. Also,
write $\bar b \equiv \bar c$ to denote $\bar
b\equiv_\emptyset\bar c$. If $X$ is a homogeneous metric
space, then the former is equivalent to the fact that there
is $g\in \iso(X)$ with $g\restriction \bar a=\id$ and
$g(\bar c)=\bar b$.

\begin{definition}\label{def:sat}
  Given a metric structure $X$, $k\in\N$, a tuple $\bar a\in
  X^k$ with $p=\mtp(\bar a)$ and $\varepsilon>0$, say that a
  subset $Y\subseteq p(X)$ is \textit{relatively
    $\varepsilon$-saturated over $\bar a$} if every
  $\varepsilon$-quantifier-free $k$-type over $\bar a$ which
  is realized in $X$, is also realized in $Y$.
\end{definition}

Note that if $Y$ is relatively $\varepsilon$-saturated over
$\bar a$, then in particular, $Y$ contains $\bar a$. Given
two tuples $\bar a,\bar b\in X^m$ and $\varepsilon>0$ write
$d_X(\bar a,\bar b)<\varepsilon$ if
$d_X(a_k,b_k)<\varepsilon$ for every $k\leq m$.

\begin{definition}
  Suppose $X$ is a homogeneous metric structure and $\bar
  a\in X^k$ for some $k\in\N$. Write $p$ for $\mtp(\bar
  a)$. Say that a sequence $(\bar a_n:n\in\N)$ of elements
  of $X^k$ is an \textit{isolated sequence in $p$} if every
  $\bar a_n$ realizes $p$ and there exists a sequence of
  $\varepsilon_n>0$ and subsets $Y_n\subseteq p(X)$ such
  that for every $n\in\N$ the set $Y_n$ is relatively
  $\varepsilon_n$-saturated over $\bar a_n$ and for every
  sequence $\bar b_n\in Y_n$ such that $\mtp(\bar
  b_n)=\mtp(\bar a_n)$ and $d_X(\bar a_n,\bar
  b_n)<\varepsilon_n$ there is an automorphism $\varphi$ of
  $X$ with $\varphi(\bar a_n)=\bar b_n$ for every $n\in\N$.
\end{definition}

The following definition will be generalized in Definition
\ref{def:misolated} below.

\begin{definition}\label{def:admits1}
  Say that a metric structure $X$ \textit{admits isolated
    sequences} if for every $k\in\N$, every complete
  quantifier-free $k$-type $p$ realized in $X$, for every
  nonmeager set $Z\subseteq p(X)$ there is a sequence $(\bar
  a_n:n\in\N)$ which is isolated in $p$ and $\bar a_n\in Z$
  for every $n$.
\end{definition}

The above definitions are enough for the study of the
Urysohn space and the measure algebra but in order to deal
with the Hilbert space we need to introduce slightly
more general notions.

\begin{definition}\label{def:rel.loc.sat}
  Given a metric structure $X$, $k\in\N$, $\varepsilon>0$
  and tuple $\bar a\in X^k$ write $p=\mtp(\bar a)$. Suppose
  $T\subseteq p(X)$. Say that $Y\subseteq X^k$ is
  \textit{$T$-relatively $\varepsilon$-saturated over $\bar
    a$} if for every $\bar b\in T$ with $d_X(\bar b,\bar
  a)<\varepsilon$ there is $\bar b'\in Y$ such that
  $\mtp(\bar b\slash\bar a)=\mtp(\bar b'\slash\bar a)$.
\end{definition}

\begin{definition}
  Suppose $X$ is a metric structure, $\varepsilon>0$,
  $k,m\in\N$, $\bar a\in X^k$ and $p=\mtp(\bar a)$. Say that
  a subset $T\subseteq p(X)$
  \textit{$(m,\varepsilon)$-generates an open set over $\bar
    a$} if there is a nonempty open set $U\subseteq p(X)$
  such that for every $\bar b\in U$ there is a sequence
  $g_1,\ldots,g_m\in\aut(X)$ such that
  \begin{itemize}
  \item $g_m\ldots g_1(\bar a)=\bar b$,
  \item $g_i(\bar a)\in T$ and $d_X(g_i(\bar a),\bar
    a)<\varepsilon$ for each $i\leq m$.
  \end{itemize}
\end{definition}

Note that, in particular, if $T$ contains an open ball of
radius $\varepsilon$ around $\bar a$ (i.e. $\{\bar b\in
p(X):d_X(\bar b,\bar a)<\varepsilon\}$), then it
$(1,\varepsilon)$-generates an open set over $\bar
a$. Therefore, the following definition is a generalization
of Definition \ref{def:sat}.

\begin{definition}
  Suppose $X$ is a metric structure, $\varepsilon>0$,
  $k,m\in\N$, $\bar a\in X^k$ and $p=\mtp(\bar a)$. Say that
  $Y\subseteq p(X)$ is \textit{$m$-relatively
    $\varepsilon$-saturated over $\bar a$} if there is
  $T\subseteq p(X)$ such that
  \begin{itemize}
  \item $Y$ is $T$-relatively $\varepsilon$-saturated over
    $\bar a$,
  \item $T$ $(m,\varepsilon)$-generates an open set over
    $\bar a$.
  \end{itemize}
\end{definition}

Thus, if $Y$ is relatively $\varepsilon$-saturated over
$\bar a$, then it is $1$-relatively $\varepsilon$-saturated
over $\bar a$.

\begin{definition}\label{def:misolated}
  Suppose $X$ is a homogeneous metric structure and $\bar
  a\in X^k$ for some $k\in\N$. Write $p$ for $\mtp(\bar
  a)$. Say that a sequence $(\bar a_n:n\in\N)$ of elements
  of $X^k$ is a \textit{weakly isolated sequence in $p$} if
  every $\bar a_n$ realizes $p$ and there exists $m\in\N$
  and a sequence of $\varepsilon_n>0$ and subsets
  $Y_n\subseteq p(X)$ such that for every $n\in\N$ the set
  $Y_n$ is $m$-relatively $\varepsilon_n$-saturated over
  $\bar a_n$ and for every sequence $\bar b_n\in Y_n$ such
  that $\mtp(\bar b_n)=\mtp(\bar a_n)$ and $d_X(\bar
  a_n,\bar b_n)<\varepsilon_n$ there is an automorphism
  $\varphi$ of $X$ with $\varphi(\bar a_n)=\bar b_n$ for
  every $n\in\N$. Given $m$ as above we will also say that
  the sequence is \textit{$m$-weakly isolated}.
\end{definition}

\begin{definition}\label{def:admits}
  Say that a metric structure $X$ \textit{admits weakly
    isolated sequences} if there is $m\in\N$ such that for
  every $k\in\N$, every complete quantifier-free $k$-type
  $p$ realized in $X$, for every nonmeager set $Z\subseteq
  p(X)$ there is a sequence $(\bar a_n:n\in\N)$ which is
  $m$-weakly isolated in $p$ and $\bar a_n\in Z$ for every
  $n$. Given $m$ as above we will also say that the
  structure admits $m$-weakly isolated sequences.
\end{definition}

Now, Definition \ref{def:admits1} is a special case of
Definition \ref{def:admits} since an isolated sequence is
obviously weakly isolated.

\section{Metric structures}\label{sec:metric.structures}

Automatic continuity for automorphism groups of metric
structures will depend on the model-theoretic properties of
the structure. The key definitions, which stem from the
analysis of the work of Kechris and Rosendal \cite{kr} are
given below. 

Below, and throughout of this paper, we use the convention
that a finitely generated substructure of a metric structure
is always enumerated (a finitely generated substructure is a
tuple if there are no function symbols).

\begin{definition}
  Let $M$ be a metric structure, $B,C\subseteq M$ be
  finitely generated substructures. Given a finitely
  generated substructure $A\subseteq B\cap C$ say that $B$
  and $C$ are \textit{independent over} $A$ and write
  $$B\ind_A C$$ if for every pair of automorphisms
  $\varphi:B\to B$, $\psi:C\to C$ such that $A$ is closed
  under $\varphi$ and $\psi$ and $\varphi\restriction
  A=\psi\restriction A$, the function $\varphi\cup\psi$
  extends to an automorphism of the substructure generated
  by $B$ and $C$.
\end{definition}

An abstract notion of stationary independence has been
considered by Tent and Ziegler in \cite{tent.ziegler}. In
general, the above notion is not a stationary independence
relation in the sense of \cite[Definition 2.1]{tent.ziegler}
and satisfies only the Invariance and Symmetry conditions
(see also the remarks \cite[Example
2.2]{tent.ziegler}). However, in all concrete cases, the
examples of independence relation considered in this paper
will be the same as in \cite{tent.ziegler}. The following is
motivated by a standard property of the independence
relation in stable theories (see \cite[Theorem
14.12]{model.theory}).

\begin{definition}
  Say that a metric structure $M$ has the \textit{extension
    property} if for every pair $B,C\subseteq M$ of finitely
  generated substructures and a finitely generated
  substructure $A\subseteq B\cap C$ there is a finitely
  generated substructure $C'\subseteq M$ with $C'\equiv_A C$
  such that $B\ind_A C'$.
\end{definition}

Another property of the metric structures that we will need
for automatic continuity is connected with the extension
theorems proved by Hrushovski \cite{hrushovski}, Herwig and
Lascar \cite{herwig.lascar} and Solecki \cite{sol.iso}.

\begin{definition}
  Say that a metric structure $M$ \textit{has locally finite
    automorphisms} if for every $n\in\N$, for any finitely
  generated substructure $N$ of $M$, for any partial
  automorphisms $\varphi_1,\ldots,\varphi_n$ of $N$, there
  is a finitely generated structure $N'$ of $M$ containing
  $N$ such that every $\varphi_i$ extends to an automorphism
  of $N'$, for each $i\leq n$.
\end{definition}

Note that if finitely generated substructures of $M$ are finite
(e.g. when there are no function symbols for functions from
$M^m$ to $M$), then $M$ has locally finite automorphisms if
and only if for any finite substructure $N$ of $M$ there is
is a finite structure $N'$ of $M$ containing $N$ such that
every isomorphism between finite substructures of $N$
extends to an automorphism of $N'$.

\section{First order logic for metric structures}\label{sec:continuous.logic}

Similarly as in the classical case, the $\elomega$-formulas
in first-order logic for metric structures are formed by
allowing countable infinite conjunctions and disjunctions of
first order formulas as well as finite quantification. Note
that the formula $\mtp(\bar x)=\mtp(\bar y)$ belongs to
$\elomega$.

A property $\Phi$ of a metric structures is called a
\textit{first-order property} if there is a set $\tilde\Phi$
of $\elomega$-sentences such that a metric structure $M$
satisfies $\Phi$ if and only if $M\models\phi$ for every
$\phi\in\tilde\Phi$. Note that if $\phi$ is an
$\elomega$-sentence and $N\prec M$, then we have that
$N\models\phi$ if and only if $M\models\phi$.

\begin{lemma}[L\"owenheim--Skolem]\label{lowenheim}
  Suppose $M$ is a metric structure and for each $\bar
  a,\bar b\in M^{<\omega}$ let $\varphi_{\bar a\bar b}$ be
  an automorphism of $M$. For every countable $M_0\subseteq
  M$ there is a countable metric structure $N\subseteq M$
  with $M_0\subseteq N$, such that $N\prec M$ and $N$ is
  closed under $\varphi_{\bar a\bar b}$ for each $\bar
  a,\bar b\in N^{<\omega}$.
\end{lemma}
\begin{proof}
  The standard L\"owenheim--Skolem argument shows that there
  is a countable $M_1$ with $M_0\subseteq M_1$ and $M_1\prec
  M$. Construct a chain of countable first-order elementary
  substructures $M_n\prec M$ with $M_n\subseteq M_{n+1}$
  such that $M_{n+1}$ is closed under $\varphi_{\bar a \bar
    b}$ for every $\bar a,\bar b\in M_n^{<\omega}$. Write
  $N=\bigcup_n M_n$. Then $N$ is as needed.
\end{proof}

\begin{lemma}\label{locally.finite.lemma}
  The property saying that a metric structure has locally
  finite automorphisms is a first-order property for
  homogeneous metric structures.
\end{lemma}
\begin{proof}
  For each $n\in\N$ write $$x\in\langle
  x_1,\ldots,x_n\rangle$$ for the $\elomega$-formula (in
  variables $x,x_1,\ldots,x_n$) saying that $x$ belongs to
  the substructure generated by $x_1,\ldots,x_n$. The
  formula is of the form
  $\bigvee_{k\in\N}x=g_k(x_1,\ldots,x_n)$ where $g_k$
  enumerate all compositions of function symbols in the
  language. We also write
  $y_1,\ldots,y_m\in\langle x_1,\ldots,x_n\rangle$ for
  $\bigwedge_{i\leq m}y_i\in\langle x_1,\ldots,x_n\rangle$.

  Note that if a homogeneous metric structure $M$ has
  locally finite automorphisms and $N$ is its finitely
  generated substructure, say by $x_1,\ldots,x_n$ and
  $k\in\N$, then there is a number $n(p,k)\in\N$, depending
  only on the quantifier-free type $p$ of $x_1,\ldots,x_n$
  and $k$ such that any for any substructure $N_1$
  isomorphic to $N$ in $M$, and any set of partial
  automorphisms $\varphi_1,\ldots,\varphi_k$ of $N_1$ there
  is a substructure $N_1'$ of $M$ containing $N_1$ and
  generated by $m\leq n(p)$ elements such that every
  $\varphi_i$ extends to an automorphism of $N_1'$.

  Thus, a homogeneous metric structure $M$ has locally
  finite automorphisms if and only if it satisfies the
  following $\elomega$-sentences, for every $n,k\in\N$,
  quantifier-free $n$-type $p$ and every $n_1,\ldots,n_k\leq
  n$. 

 \begin{eqnarray*}
   \forall x_1,\ldots,x_n\quad [\cmtp(x_1,\ldots,x_n)=p]\Rightarrow\\
   \forall y_1^1,\ldots,y^1_{n_1},z^1_1,\ldots,z^1_{n_1}
   y_1^2,\ldots,y^2_{n_2},z^2_1,\ldots,z^2_{n_2},
   \ldots,
   y_1^k,\ldots,y^k_{n_k},z^k_1,\ldots,z^k_{n_k}\\
   \bigg[\bigwedge_{i=1}^k\big(\bigwedge_{l=1}^{n_i}
   y_l^i,z_l^i\in\langle x_1,\ldots,x_n\rangle\big)
   \,\wedge\,\mtp(y_1^i,\ldots,y^i_{n_i})=\mtp(z^i_1,\ldots,z^i_{n_i})\bigg]\\
   \Rightarrow\quad\bigg[\bigvee_{n\leq m\leq n(p,k)}\exists x_1',\ldots,x_m'\quad
   x_1'=x_1\wedge\ldots\wedge x_n'=x_n\\
   \bigwedge_{j\leq k}\,\exists x_1^k,\ldots,x_m^k\in\langle
   x_1',\ldots,x_m'\rangle\quad \bigg(x_1',\ldots,x_m'\in\langle x_1^k,\ldots,x_m^k\rangle\\
   \wedge\quad
   \mtp(x_1',\ldots,x_m',y_1^k,\ldots,y_{n_k}^k)=\mtp(x_1^k,\ldots,x_m^k,z_1^k,\ldots,z_{n_k}^k)\bigg)\bigg]
  \end{eqnarray*}
\end{proof}

\begin{lemma}\label{independence.lemma}
  The extension property is a first-order property for
  metric structures.
\end{lemma}
\begin{proof}
  The extension property is the conjunction of the following
  sentences, for all $n,m\in\N$ and $k\leq\min(n,m)$. Below,
  for a tuple $\bar x=(x_1,\ldots,x_n)$ and $\sigma\in S_n$
  (the group of permutations of $n$) we write $\bar
  x_\sigma$ for the tuple
  $(x_{\sigma(1)},\ldots,x_{\sigma(n)})$.
 \begin{eqnarray*}
   \forall x_1,\ldots,x_n\quad\forall
   y_1,\ldots,y_m\quad(x_1=y_1\wedge\ldots\wedge
   x_k=y_k)\\
   \Rightarrow\bigg[\exists y_1'\ldots,y_m'
   (y_1'=y_1\wedge\ldots\wedge y_k'=y_k) \wedge \mtp(y_1\ldots
   y_m)=\mtp(y_1'\ldots y_m')\\\wedge\ \bigwedge_{\sigma\in
     S_n}\bigwedge_{\tau\in S_m} \Big(\sigma\restriction k:k\to
   k\quad\wedge\quad \sigma\restriction k=\tau\restriction
   k\\ \big(\mtp(\bar x)=\mtp(\bar x_\sigma)\ \wedge
   \ \mtp(\bar y')=\mtp(\bar y'_\tau)
   \big)
   \Rightarrow \mtp(\bar x,\bar y)=\mtp(\bar
   x_\sigma,\bar y_\tau)\Big)\bigg]
  \end{eqnarray*}
\end{proof}

\begin{corollary}\label{lowenheimcor}
  Suppose $M$ is a homogeneous metric structure which has
  locally finite automorphisms and the extension
  property. Let $\varphi_{\bar a}$ be an automorphism of $M$ for
  each $\bar a\in M^{<\omega}$. Then there is a countable
  homogeneous metric structure $N\subseteq M$ which is dense
  in $M$, has locally finite automorphisms, the extension
  property and is closed under all automorphisms
  $\varphi_{\bar a}$ for $\bar a\in N^{<\omega}$.
\end{corollary}
\begin{proof}
  Let $M_0$ be countable dense in $M$ and for each $\bar
  a,\bar b\in M^{<\omega}$ which generate isomorphic
  substructures, let $\varphi_{\bar a\bar b}$ be an
  automorphism of $M$ which maps $\bar a$ to $\bar b$. If
  $\bar a,\bar b\in M^{<\omega}$ are of different
  cardinality or do not generate isomorphic substructures,
  then put $\varphi_{\bar a\bar b}=\varphi_{\bar a}$. By
  Proposition \ref{lowenheim}, there is a countable
  substructure $N\prec M$ which contains $N_0$ and is closed
  under all $\varphi_{\bar a\bar b}$ for $\bar a,\bar b\in
  N^{<\omega}$. The latter clearly implies that $N$ is
  homogeneous and closed under $\varphi_{\bar a}$ for $\bar
  a\in N^{<\omega}$. Lemmas \ref{independence.lemma} and
  \ref{locally.finite.lemma} imply that $N$ has the
  extension property and locally finite automorphisms.
\end{proof}

\section{Ample generics}\label{sec:ample.generics}

If a metric structure is countable and has a discrete
metric, then its automorphism group is a subgroup of
$S_\infty$ and for such groups Kechris and Rosendal
\cite{kr} developed machinery for proving automatic
continuity. Recall that a topological group $G$ has
\textit{ample generics} if for every $n\in\N$ there is a
comeager class in the diagonal conjugacy action of $G$ on
$G^n$ (i.e. the action
$g\cdot(g_1,\ldots,g_n)=(gg_1g^{-1},\ldots,gg_ng^{-1})$).

Recall \cite{hjorth} that given a continuous action of a
Polish group $G$ on a Polish space $X$, a point $x\in X$ is
\textit{turbulent} if for every open neighborhood
$U\subseteq G$ of the identity and every open $V\ni x$, the
\textit{local orbit} $O(x,U,V)=\{x'\in X:\exists
g_1,\ldots,g_n\in U \forall i\leq n \ g_i\ldots g_1x\in
U\wedge x'=g_n\ldots g_1x\}$ is somewhere dense. If $G$ is a
subgroup of $S_\infty$, then a point $x\in X$ is turbulent
if and only if for every open subgroup $U\leq G$ the set
$Ux=\{gx:g\in U\}$ is somewhere dense \cite[Proposition
3.2]{kr}. Also, in the case of a continuous action of a
Polish group, if a point is turbulent and has a dense orbit,
then its orbit is actually comeager. This is because in such
case the orbit of a turbulent point cannot be meager and
hence has to be meager in its closure \cite[Proposition
3.2]{kr}.

The groups of automorphisms of metric structures can be
endowed with many topologies and this is the starting point
of the analysis of Ben Yaacov, Berenstein and Melleray
\cite{bybm}, who consider two types of topologies: the
Polish topologies of pointwise convergence and variants of
strong (non separable) topologies. We are, however,
primarily interested in separable topologies on these
groups. By default, the topology on $\aut(M)$ is that of
pointwise convergence with respect to the metric on
$M$. However, if $M$ is countable (but perhaps not complete
with respect to its metric), we will also consider the
topology inherited from the group $S_\infty$, which
coincides with the pointwise convergence topology with
respect to a discrete metric on $M$. If $M$ is countable,
then we refer to this topology by saying that the group
$\aut(M)$ is treated \textit{as a subgroup of
  $S_\infty$}. It can be also viewed as the topology of
pointwise convergence on the automorphism group of the
structure $M$ endowed with a discrete metric and the
original distance function $d_M$ treated as a part of the
structure.

Below, and throughout this paper, we use the following
notation. If $G$ acts on $M$ and $\bar a=(a_1,\ldots,a_n)\in
M^{<\omega}$, then $G_{\bar a}=\{g\in G:\forall i\leq n\ 
g(a_i)=a_i\}$.

\begin{lemma}\label{amplegen}
  Suppose $M$ is a countable homogeneous metric
  structure. If $M$ has locally finite automorphisms and the
  extension property, then $\aut(M)$ has ample generics as a
  subgroup of $S_\infty$.
\end{lemma}
\begin{proof}
  Write $G$ for $\aut(M)$. Fix $n\in\N$ to find a generic
  tuple in $G^n$. Enumerate as $(a_n:n\in\N)$ with infinite
  repetitions all tuples $a=(A,\vec\varphi,B,\vec\psi)$ with
  $\vec\varphi=(\varphi_1,\ldots,\varphi_n)$,
  $\vec\psi=(\psi_1,\ldots,\psi_n)$ such that $A\subseteq B$
  are finitely generated substructures of $M$ (possibly
  generated by the empty set) and for each $i\leq n$ we have
  $\varphi_i\subseteq\psi_i$, and $\varphi_i:A\to A$,
  $\psi_i:B\to B$ are automorphisms.

  By induction on $k$ construct a sequence of increasing
  finitely generated substructures $D_k\subseteq M$ together
  with tuples of increasing automorphisms
  $\vec\gamma_k=(\gamma_k^1,\ldots,\gamma_k^n)$ with
  $\gamma_k^i:D_k\to D_k$ for each $i\leq n$. Using
  back-and-forth and homogeneity make sure that $\bigcup_k
  D_k=M$ and for each $i\leq n$ the function
  $\bigcup_k\gamma_k^i$ is an automorphism of
  $M$. Additionally, for each $k$ make sure that
  \begin{equation}\label{eq:ample}
    \begin{split}
      &\mbox{if } a_k=(A_k,\vec\varphi_k,B_k,\vec\psi_k)
      \mbox{ and }
      A_k\subseteq D_k \mbox{ are such that }\\
      &\vec\varphi_k=(\varphi_k^1,\ldots,\varphi_k^n),\ 
      \vec\psi_k=(\psi_k^1,\ldots,\psi_k^n),\\ &
      \varphi^i_k\restriction A_k:A_k\to A_k,\ 
      \psi^i_k\restriction A_k=\varphi^i_k\restriction A_k \mbox{ and }
      \varphi_k^i\subseteq\gamma_k^i \mbox{ for each } i\leq n,\\
      &\mbox{then there is } g\in G_{A_k} \mbox{ with }
      g\gamma_{k+1}^i g^{-1}\supseteq\psi_k^i \mbox{ for
        each } i\leq n.
    \end{split}\tag{$*$}
  \end{equation}
  At the induction step $k$ first use locally finite
  automorphisms to find $D_k'\supseteq D_k$ such that the
  first $k$-many elements of $M$ are in $D_k'$ and each
  $\varphi_k^i$ extends to an automorphism ${\gamma_k^i}'$
  of $D_k'$. Next, use the extension property to find a
  finitely generated substructure $B_k'$ of $M$ with
  $B_k'\equiv_{A_k}B_k$ such that $B_k'\ind_{A_k} D_k'$. Let
  $g\in G_{A_k}$ witness that $B_k'\equiv_{A_k}B_k$,
  i.e. $g(B_k')=B_k$. Define $D_{k+1}$ to be the
  substructure generated by $D_k'$ and $B_k'$ and for each
  $i\leq n$ define $\gamma_{k+1}^i$ so that
  $\gamma_{k+1}^i\restriction D_k'={\gamma_k^i}'$ and
  $\gamma_{k+1}^i\restriction B_k'=g^{-1}\psi_ig$ and use
  the assumption $B_k'\ind_{A_k} D_k'$ to extend it to
  $D_{k+1}$. Note that $g$ witnesses that (\ref{eq:ample})
  is satisfied at the step $k$.

  After this construction is done, write
  $g_i=\bigcup_k\gamma_k^i$ and note that $g_i$ is an
  automorphism of $M$, for each $i\leq n$. To see that $\vec
  g=(g_1,\ldots,g_n)$ is generic, it is enough to see that
  $\vec g$ is turbulent under the diagonal conjugacy action
  of $\aut(M)$ on $\aut(M)^n$ and has a dense orbit. 

  To see the turbulence, fix an open neighborhood $O$ of the
  identity in $\aut(M)$, and, say, $O=G_{\bar a}$ for a
  finite tuple $\bar a\subseteq M$.  We need to see that
  $O\cdot\vec g$ is somewhere dense. Find $m$ such that
  $\bar a\subseteq D_m$ and write $V_i$ for $\{f\in\aut(M):
  f\restriction D_m=g_i\restriction D_m\}$. Note that
  $g_i\restriction D_m:D_m\to D_m$ is an automorphism for
  each $i\leq n$. We claim that $O\cdot\vec g$ is dense in
  $V_1\times\ldots\times V_n$. To see that, fix a nonempty
  open subset $W\subseteq V_1\times\ldots\times V_n$ and
  without loss of generality assume $W=W_1\times\ldots\times
  W_n$. Since $M$ has locally finite automorphisms, we can
  assume that $W_i=\{f\in\aut(M):f\supseteq\psi_i\}$ for
  each $i\leq n$, where each $\psi_i:B\to B$ is an
  automorphism of a finitely generated substructure $B$ of
  $M$. Note that $\psi_i\supseteq g_i\restriction D_m$ for
  each $i\leq n$.

  We need to see that $(O\cdot\vec g)\cap
  (W_1\times\ldots\times W_n)\not=\emptyset$. Let $k\in\N$
  be such that $a_k=(A_k,\vec\varphi_k,B_k,\vec\psi_k)$ with
  $A_k=D_m$, $B_k=B$, $\vec\varphi_k=(g_1\restriction
  D_m,\ldots,g_n\restriction D_m)$ and
  $\vec\psi=(\psi_1,\ldots,\psi_n)$. By (\ref{eq:ample}) at
  the step $k$, there is $g\in G_{A_k}\subseteq G_{\bar a}$
  such that $gg_ig^{-1}\supseteq\psi_i$ for each $i\leq n$
  and so we are done.

  The fact that the orbit of $\vec g$ is dense follows by
  analogous arguments.
\end{proof}

Recall that a subset $S$ of a group $G$ is called
\textit{countably syndetic} if $G$ can be written as
$\bigcup_n a_nS$ for some sequence $a_n\in G$. A subset $S$
of $G$ is called \textit{symmetric} if $S^{-1}=S$.

\begin{corollary}\label{amplegencor}
  Suppose $M$ is a countable homogeneous metric structure
  and $G=\aut(M)$. If $M$ has locally finite automorphisms
  and the extension property, and $W\subseteq\aut(M)$ is
  symmetric and countably syndetic, then there is a finite
  tuple $\bar a\subseteq M$ such that $G_{\bar a}\subseteq
  W^{10}$.
\end{corollary}
\begin{proof}
  This is an abstract consequence of ample generics
  \cite[Lemma 6.15]{kr} and thus follows from Lemma
  \ref{amplegen}.
\end{proof}

\section{Automatic continuity}\label{sec:automatic.continuity}

In \cite{rosendal.solecki} Rosendal and Solecki isolated an
abstract property of a group that implies automatic
continuity. We say that $G$ \textit{is Steinhaus}
(cf. \cite{steinhaus}) if there is a natural number $k\geq
1$ such that for every symmetric countably syndetic set
$S\subseteq G$ the set $S^k=\{s_1\cdot\ldots\cdot
s_k:s_1,\ldots,s_k\in S\}$ contains a nonempty open set in
$G$. In such a case $G$ is also called
\textit{$k$-Steinhaus}. If a group $G$ is Steinhaus, then
$G$ has the automatic continuity property \cite[Proposition
2]{rosendal.solecki}.

We will need the following lemma.

\begin{lemma}\label{elementary}
  Suppose $X$ is a complete homogeneous metric structure
  that has locally finite automorphisms and the extension
  property and let $G=\aut(X)$. If $W\subseteq G$ is
  symmetric and countably syndetic, then there is $\bar a\in
  X^{<\omega}$ such that $G_{\bar a}\subseteq W^{10}$.
\end{lemma}
\begin{proof}
  Suppose otherwise. This means that for each $k\in\N$ and
  $a\in X^k$ there is $f_{\bar a}\in G$ with $f_{\bar
    a}(\bar a)=\bar a$ and $f_{\bar a}\notin W^{10}$. Let $g_n\in G$ be such that
  $G=\bigcup_n g_nW$. Use
  Corollary \ref{lowenheimcor} to find a countable, dense
  $X_0\subseteq X$ such that $X_0\prec X$ and $X_0$ is
  closed under each $g_n$ as well as under $f_{\bar a}$ for
  each $\bar a\in(X_0)^{<\omega}$. Since $X_0$ is elementary
  in $X$, it has locally finite automorphisms and the
  extension property by Lemmas \ref{locally.finite.lemma}
  and \ref{independence.lemma}.

  Write $W_0$ for the set of those automorphisms of $X_0$
  whose unique extension to $X$ belongs to $W$.

  \begin{claim}
    $W_0$ is symmetric and countably syndetic in
    $\aut(X_0)$.
  \end{claim}
  \begin{proof}
    It is clear that $W_0$ is symmetric. To see that it is
    countably syndetic, pick $f_0\in\aut(X_0)$ and let
    $f\in\aut(X)$ be the unique extension of $f_0$ to
    $X$. Now, there is $n\in\N$ and $s\in W$ such that
    $f=g_ns$. Since $s=g_n^{-1}f$ leaves $X_0$ invariant, we
    have that $s_0=s\restriction X_0\in\aut(X_0)$ and
    $s_0\in W_0$.
  \end{proof}

  Now, this gives a contradiction since by Corollary
  \ref{amplegencor}, $(W_0)^{10}$ contains an open
  neighborhood of the identity in $\aut(X_0)$, i.e. there is
  $\bar a\in (X_0)^{<\omega}$ such that every automorphism
  of $X_0$ which fixes $\bar a$ belongs to
  $(W_0)^{10}$. Write $f_0$ for $f_{\bar a}\restriction X_0$
  and note that $f_0\notin(W_0)^{10}$, by the density of
  $X_0$, contradiction. This proves the lemma.
\end{proof}

\begin{theorem}
  Suppose $X$ is a complete homogeneous metric structure
  that admits weakly isolated sequences, has locally finite
  automorphisms and the extension property. Then the group
  $\aut(X)$ is Steinhaus.
\end{theorem}

\begin{proof}
  Write $G$ for $\aut(X)$. Suppose $m\in\N$ is such that $X$
  admits $m$-weakly isolated sequences. We will show that
  $G$ is $(24m+10)$-Steinhaus. The same argument also shows
  that if $X$ admits isolated sequences then $G$ is
  $24$-Steinhaus. Let $W\subseteq G$ be symmetric and
  countably syndentic. Let $g_n\in G$ be such that
  $G=\bigcup_n g_nW$.

  Let $\bar a$ be such as in Lemma \ref{elementary} and
  $k\in\N$ such that $\bar a\in X^k$. Write $\bar
  a=(a_1,\ldots,a_k)$, $p$ for $\mtp(\bar a)$ and
  $Z=\{w(\bar a):w\in W\}\subseteq p(X)$. Note that since
  $p(X)=\bigcup_n g_nZ$, the set $Z$ is nonmeager in
  $p(X)$. Choose an $m$-weakly isolated sequence $\bar a_n$
  such that each $\bar a_n$ belongs to $Z$ and for each
  $n\in\N$ let $v_n\in W$ be such that $\bar a_n=v_n(\bar
  a)$. Let $\varepsilon_n>0$, and $T_n,X_n\subseteq p(X)$
  witness that the sequence of $\bar a_n$ is $m$-weakly
  isolated (i.e. $X_n$ is $T_n$-relatively
  $\varepsilon_n$-saturated in $X$ and $T_n$
  $(m,\varepsilon_n)$-generates an open set over $\bar
  a_n$).

  Given two subspaces $X',X''$ of $X$ and a set $C$ of
  partial automorphisms from $X'$ to $X''$, say that a
  subset $G_0$ of $G$ is \textit{full for $C$} if every
  element of $C$ can be extended to an element of $G_0$
  (cf. \cite[Claim 1]{rosendal.solecki}).

  \begin{claim}\label{fullness}
    There is $n\in\N$ such that $W^2$ is full
    for $$C_n=\{\varphi:\bar a_n\to\ball_X(\bar
    a_n,\varepsilon_n)\cap X_n:\varphi\mbox{ is an
      isomorphic embedding}\}.$$
  \end{claim}
  \begin{proof}
    First note that there is $n\in\N$ such that $g_nW$ is
    full for $C_n$. If not, then for each $n\in\N$ there is
    $\varphi_n\in C_n$ such that $\varphi_n$ cannot be
    extended to an element of $g_nW$. Since $\bar a_n$ are
    isolated, there is an automorphism $\varphi$ of $X$
    which extends all the $\varphi_n$. Then
    $\varphi\notin\bigcup_n g_nW$, which is a
    contradiction. Now, if $g_nW$ is full for $C_n$, then so
    is $W^2=(g_nW)^{-1}(g_nW)$ as $C_n$ contains the
    identity.
  \end{proof}

  Fix $n$ as in Claim \ref{fullness} and write
  $T=v_n^{-1}(T_n)$.
  \begin{claim}\label{additional}
    We have $$\{g\in G: d_X(\bar a,g(\bar
    a))<\varepsilon_n\mbox{ and }g(\bar a)\in T\}\subseteq
    W^{24}.$$
  \end{claim}
  \begin{proof}
    Let $g\in G$ be such that $g(\bar a)\in T$ and
    $d_X(a_i,g(a_i))<\varepsilon_n$ for each $i\leq k$.

    Let $Y=v_n^{-1}(X_n)$. Note that since $v_n$ is an
    automorphism and $X_n$ is $T_n$-relatively
    $\varepsilon_n$-saturated over $\bar a_n$, we get that
    $Y$ is $T$-relatively $\varepsilon_n$-saturated over
    $\bar a$. Thus, there is $\bar b\in Y$
    with $$\mtp(\bar b\slash\bar a)=\mtp(g(\bar a)\slash
    \bar a).$$ By homogeneity of $X$, there is $w_1\in
    G_{\bar a}$ such that $w_1(g(\bar a))=\bar b$. Note that
    $w_1\in W^{10}$ as $G_{\bar a}\subseteq W^{10}$.

    Look at $v_nw_1^{-1}gv_n^{-1}\in G$ and note that it
    maps $\bar a_n$ to $\ball_X(\bar a_n, \varepsilon_n)\cap
    X_n$ because $v_n$ maps $Y$ to $X_n$. By Claim
    \ref{fullness} and the choice of $n$, there is $w_2\in
    W^2$ which is equal to $v_{n}w_1^{-1}gv_n^{-1}$ on $\bar
    a_n$. This means that $w_2^{-1}v_{n}w_1^{-1}gv_n^{-1}\in
    G_{\bar a_n}$ and thus
    $v_n^{-1}w_2^{-1}v_{n}w_1^{-1}g\in G_{\bar
      a}$. Therefore, $v_n^{-1}w_2^{-1}v_{n}w_1^{-1}g\in
    W^{10}$ and thus $g\in W^{24}$. 
  \end{proof}
  
  Now note that $T$ $(m,\varepsilon_n)$-generates an open
  set over $\bar a$, so let $U\subseteq p(X)$ be nonempty
  and such that for every $\bar b\in U$ there is a sequence
  $g_1,\ldots,g_m\in\aut(X)$ such that
  \begin{itemize}
  \item $g_m\ldots g_1(\bar a)=\bar b$,
  \item $g_i(\bar a)\in T$ and $d_X(g_i(\bar a),\bar
    a)<\varepsilon_n$ for each $i\leq m$.
  \end{itemize}
  We claim that $$\{g\in G: g(\bar a)\in U\}\subseteq
  W^{24m+10}.$$ To see this, let $g\in G$ be such that
  $g(\bar a)\in U$ and let $\bar b=g(\bar a)$. Find
  $g_1,\ldots,g_m$ as above and note that by Claim
  \ref{additional} we have that $g_i\in W^{24}$ for each
  $i\leq m$. Now, $g^{-1}g_m\ldots g_1\in G_{\bar a}$, so
  $g\in W^{24m+10}$, as needed.

  The set $\{g\in G: g(\bar a)\in U\}$ is open in $G$, so we
  have that $G$ is $(24m+10)$-Steinhaus. This ends the proof.
\end{proof}

\section{Triviality of homomorphisms}\label{sec:trivial}

In this section we study the circumstances under which one
can exclude nontrivial homomorphisms from the groups of the
form $\aut(M)$ to certain topological groups $H$. Given two
groups $G$ and $H$ say that $G$ is \textit{$H$-trivial} if
any homomorphism from $G$ to $H$ is trivial. Tsankov
\cite{tsankov} concluded (from the minimality of the unitary
group) that whenever $H$ is a separable group which admits a
complete left-invariant metric, then the unitary group is
$H$-trivial. In this section, we isolate an abstract
property of a metric structure $M$ which implies that
$\aut(M)$ is $H$-trivial for $H$ as above and in Section
\ref{sec:hilbert} we will see that this property is
satisfied by the Hilbert space. The same is true for the
group $\aut(X,\mu)$ but here it follows immediately from the
fact that $\aut(X,\mu)$ is simple \cite{fathi}. The unitary
group is not simple but (similarly as the group of
isometries of the Urysohn space \cite{tent.ziegler}) has a
maximal proper normal subgroup \cite{fong}. We do not know,
however, if the methods below apply to $\U$.

Given a subset $N$ of a metric structure $M$ and $\bar a\in
M^{<\omega}$, say that $N$ is \textit{relatively saturated
  over $\bar a$} if it is $\alpha$-relatively saturated over
$\bar a$ for every $\alpha\in[0,\infty)$.

\begin{definition}
  Suppose $M$ is a homogeneous metric structure and $\bar
  a\in M^k$ for some $k\in\N$. Write $p$ for $\mtp(\bar
  a)$. Say that a sequence $(\bar a_n:n\in\N)$ of elements
  of $M^k$ is an \textit{independent sequence in $p$} if
  every $\bar a_n$ realizes $p$ and there exists a sequence
  of subsets $N_n$ such that $N_n$ is relatively saturated
  over $\bar a_n$ such that for every sequence $\bar b_n\in
  N_n$ with $\mtp(\bar b_n)=\mtp(\bar a_n)$ there is an
  automorphism $\varphi$ of $M$ with $\varphi(\bar a_n)=\bar
  b_n$ for every $n\in\N$.
\end{definition}

\begin{definition}
  Say that a metric structure $M$ \textit{admits independent
    sequences} if for every $k\in\N$ and $\bar a\in M^k$,
  for every sequence $(\bar s_n:n\in\N)$ of finite tuples of
  elements of $M$, there is a sequence $\bar a_n$ which is
  independent in $\mtp(\bar a)$ and is such that $\bar
  a_{n+1}\equiv_{\bar s_n} \bar a_n$ holds for each
  $n\in\N$.
\end{definition}

\begin{theorem}
  Suppose $X$ is a complete metric structure that admits
  independent sequences, has locally finite automorphisms
  and the extension property. Then the group $\aut(X)$ is
  $H$-trivial for every Polish group $H$ which has a
  complete left-invariant metric.
\end{theorem}
\begin{proof}
  Write $G$ for $\aut(X)$ and suppose $H$ has a complete
  left-invariant metric. Fix $\varphi: G\to H$ and we will
  show that $\varphi$ is trivial. To do this, it is enough
  to see that whenever $U\subseteq H$ is an open
  neighborhood of the identity, then $\varphi(G)\subseteq
  U$.

  Let then $U\subseteq H$ be an open neighborhood of the
  identity. Find $V\subseteq H$ open neighborhood of the
  identity such that $V^{22}\subseteq U$ and write
  $T=\varphi^{-1}(V)$.  Note that $T\subseteq G$ is
  countably syndetic, so by Lemma \ref{elementary}, there is
  $k\in\N$ and $\bar a\in X^k$ such that $G_{\bar
    a}\subseteq T^{10}$. Write $p$ for the quantifier-free
  type of $\bar a$.

  Let $d_H$ be a complete left-invariant metric on $H$ and
  pick a decreasing sequence of open neighborhoods $V_n$ of
  the identity in $H$ with $V_0=V$ such that
  $\diam_{d_H}(V_n)\leq2^{-n}$. Let also $W_n\subseteq H$ be
  open symmetric neighborhoods of the identity in $H$ with
  $(W_n)^{10}\subseteq V_n$. Note that each
  $\varphi^{-1}(W_n)$ is countably syndetic in $G$. Using
  Lemma \ref{elementary}, by induction on $n>0$ pick
  increasing sequence $\bar s_n$ of finite tuples of
  elements of $X$ such that $G_{\bar
    s_n}\subseteq\varphi^{-1}(W_n)^{10}\subseteq\varphi^{-1}(V_n)$. Let
  $\bar s_0$ be the empty tuple.

  Write $\bar a_0=\bar a$. Using the assumption that $X$
  admits independent sequences, pick a sequence $g_n\in G$
  for $n>0$ such that $g_n\in G_{\bar s_n}$ and the sequence
  $(\bar a_n: n>0)$ defined as $a_{n+1}=g_n(\bar a_n)$,
  forms an independent sequence in $p$. Let $X_n\subseteq X$
  witness the that the sequence is independent, so that each
  $X_n$ is relatively saturated over $\bar a_n$.

  Let $f_n=g_n\ldots g_0$ and write $h_n=\varphi(f_n)$. Note
  that $h_n^{-1}$ is $d_H$-Cauchy in $H$ and hence
  convergent. Thus, $h_n$ is convergent in $H$ too. Let
  $h=\lim_n h_n$.

  \begin{claim}\label{fullness1}
    There is $n$ such that $f_nT^2f_n^{-1}$ is full
    for $$C_n=\{\varphi:\bar a_n\to X_n:\varphi\mbox{ is an
      isomorphic embedding}\}.$$
  \end{claim}
  \begin{proof}
    We will first prove that there is a sequence $b_n\in G$
    such that $G=\bigcup_n b_nTf_n^{-1}$. This will follow
    from the fact that there is a sequence $a_n\in H$ such
    that $H=\bigcup_n a_nVh_n^{-1}$ by taking $b_n$ such
    that $\varphi(b_n)=a_n$.

    To see the latter fact, pick $a_n$ so that they are
    dense in $H$ and we claim that $H=\bigcup_n
    a_nVh_n^{-1}$. Indeed, if $x\in H$, then note that the
    sequence $$z_n=xh_n=x\varphi(g_n)\ldots\varphi(g_0)$$ is
    convergent in $H$ to $z=xh$. Pick a subsequence
    $a_{k_n}$ converging to $z$. Since $z_{k_n}$ converges
    to $z$, we get that $d_H(z_{k_n},a_{k_n})\to 0$. Thus,
    $d_H(a_{k_n}^{-1}z_{k_n},1_H)\to 0$ as well and there is
    $n$ such that $a_{k_n}^{-1}xh_{k_n}\in V$ and then $x\in
    a_{k_n}Vh_{k_n}^{-1}$.

    Now, note that there is $n\in\N$ such that $b_n T
    f_n^{-1}$ is full for $C_n$. If not, then for each
    $n\in\N$ there is $\varphi_n\in C_n$ such that
    $\varphi_n$ cannot be extended to an element of
    $b_nTf_n^{-1}$. As the sequence $\bar a_n$ is
    independent, there is an automorphism $\varphi$ of $X$
    which extends all the $\varphi_n$. But then
    $\varphi\notin\bigcup_n a_nTf_n^{-1}$, which is a
    contradiction. Finally, if $a_nTf_n^{-1}$ is full for
    $C_n$, then so is
    $f_nT^2f_n^{-1}=(a_nTf_n^{-1})^{-1}(a_nTf_n^{-1})$ since
    $C_n$ contains the identity. This proves the claim.
  \end{proof}

  We need to prove that $G\subseteq \varphi^{-1}(U)$. Let
  $g\in G$ be arbitrary. Fix $n$ as in Claim \ref{fullness1}
  and note that $f_n(\bar a)=\bar a_n$. Let
  $Y=f_n^{-1}(X_n)$. Note that $Y$ and $X$ realize the same
  quantifier-free $n$-types over $\bar a$. This follows from
  the fact that $X_n$ is relatively saturated over $\bar
  a_n$ and $f_n$ is an automorphism.

  Thus, there is $\bar b\in Y$ such that $$\mtp(\bar
  b\slash\bar a)=\mtp(g(\bar a)\slash \bar a).$$ Let $w_1\in
  G_{\bar a}$ be such that $w_1(g(\bar a))=\bar b$. Note
  that $w_1\in T^{10}$ as $G_{\bar a}\subseteq W^{10}$.

  Look at $f_{n}w_1^{-1}gf_n^{-1}\in G$ and note that it
  maps $\bar a_n$ to $X_n$ as $f_n$ maps $Y$ to $X_n$. By
  Claim \ref{fullness1}, there is $y\in f_nT^2f_{n}^{-1}$
  which is equal to $f_{n}w_1^{-1}gf_n^{-1}$ on $\bar
  a_n$. Write $y=f_nw_2f_n^{-1}$ for some $w_2\in W^2$ and
  note
  that $$y^{-1}(f_{n}w_1^{-1}gf_n^{-1})=f_nw_2w_1^{-1}gf_n^{-1}\in
  G_{\bar a_n}.$$ Therefore,
  $w_2w_1^{-1}g=f_n^{-1}(y^{-1}(f_{n}w_1^{-1}gf_n^{-1}))f_n$
  fixes $\bar a$, which means that $$w_2w_1^{-1}g\in T^{10}$$
  and thus $g\in T^{22}\subseteq\varphi^{-1}(U)$.

  This shows that $G\subseteq\varphi^{-1}(U)$ and since
  $U\subseteq H$ was an arbitrary open neighborhood of the
  identity, we have that $\varphi$ is trivial. This ends the
  proof.
\end{proof}

\section{The Urysohn space}\label{sec:urysohn}

There is no essential difference in verifying that the
Urysohn space and the Urysohn sphere satisfy the assumptions
of Theorem \ref{abstract}. We will focus only on the Urysohn
space.

Locally finite automorphisms for the Urysohn space have been
already shown by Solecki \cite{sol.iso}, who proved that for
any finite metric space $X$ and finitely many partial
isometries of $X$, there is a metric space $Y$ containing
$X$ such that all these partial isometries extend to
isometries of $Y$. Solecki derived his result from an
extension theorem of Herwig and Lascar
\cite{herwig.lascar}. The theorem of Herwig and Lascar is
connected to the Rhodes' Type II Conjecture proved
independently by Ash \cite{ash} and by Ribes and
Zalesski\v\i\ \cite{ribes.zalesskii}. The latter results
concern the profinite topology on the free groups (cf. also
\cite{pin.reutenauer}).

Recall that \textit{the profinite topology} on a free group
$F_n$ is the one with the basis at the identity consisting
of finite-index subgroups of $F_n$. In the literature, the
fact that a set $A\subseteq F_n$ is closed in the profinite
topology is usually referred to as by saying that $A$ is
\textit{separable}. A classical result of M. Hall, Jr.
\cite{hall} says that any finitely generated subgroup of
$F_n$ is separable. Note that it also implies that any coset
of a finitely generated subgroup of $F_n$ is separable
(since the multiplication is continuous in the profinte
topology). The main result of Ribes and Zalesski\v\i\
\cite{ribes.zalesskii} states that products of finitely many
finitely generated subgroups of $F_n$ are also
separable. Again, note that it immediately implies that
products of finitely many cosets of finitely generated
subgroups of $F_n$ are separable as well.

An abstract connection between the theorem of Ribes and
Zalesski\v\i\ and extensions of partial isometries was
discovered by Rosendal \cite{rosendal.ribes.zalesskii}, who
expressed it in the language of finitely approximable
actions and, in particular, gave a new proof of the result
of Solecki \cite{sol.iso}. On the other hand, the paper of
Solecki \cite{sol.iso} contains a very elegant argument on
the extensions of one isometry. That argument is done in the
style of Mackey's constructions of induced actions
\cite[Page 190]{mackey} (cf. \cite[2.3.5]{becker.kechris})
and a similar argument has been used by Hrushovski
\cite{hrushovski} in the context of extensions of partial
isomorphisms of graphs.

Below, we present a new proof of Solecki's theorem
\cite{sol.iso}, which exploits the ideas used in the case of
one isometry in \cite[Section 3]{sol.iso} and is also done
in the style of Mackey's construction of induced actions.

\begin{theorem}[Solecki]
  The Urysohn space has locally finite automorphisms.
\end{theorem}
\begin{proof}[Proof (\`a la Mackey)]
  By the finite extension property of the Urysohn space, it
  is enough to show that for every finite metric space $X$,
  for every tuple $\varphi_1,\ldots,\varphi_n$ of partial
  isometries of $X$ there is a finite metric space
  $Y\supseteq X$ such all $\varphi_1,\ldots,\varphi_n$
  extend to isometries of $Y$.

  Let $X$ be a finite metric space. Write
  $$\delta=\min\{d_X(x,y):x,y\in X,x\not=y\}$$ and let
  $\Delta=\diam(X)$. Suppose $\varphi_1,\ldots,\varphi_n$
  are partial isometries of $X$.

  Write $a_1,\ldots,a_n$ for the generators of the free
  group $F_n$. Write also $A$ for the set
  $\{a_1,\ldots,a_n,a_1^{-1},\ldots,a_n^{-1}\}$ and $W_n$
  for $A^*$ (the set of all words over $A$). For a word
  $w\in W_n$ with $w=v_1\ldots v_k$, $v_i\in A$ and $x\in X$
  say that $w(x)$ \textit{is defined} if there is a sequence
  of points $x_j\in X$ ($j\leq k$) with $x_0=x$ and
  $x_{j+1}=\varphi_i(x_j)$ if $v_{k-j}=a_i$ and
  $x_{j+1}=\varphi_i^{-1}(x_j)$ if $v_{k-j}=a_i^{-1}$. If
  $w(x)$ is defined, then write $w(x)=y$ for $y=x_k$ as
  above. We also use the notation $w(x)$ if $w$ belongs to
  $F_n$ (using the reduced word for $w$).

  For each $x,y\in X$ write $T_x^y$ for the set of $w\in
  F_n$ such that $w(x)=y$.

  \begin{claim}\label{cosets}
    For every $x,y\in X$ the set $T_x^y$ is either empty or
    a coset of a finitely generated subgroup of $F_n$.
  \end{claim}
  \begin{proof}
    If $T_x^y$ is nonempty, then let $w\in F_n$ be such that
    $w(x)=y$. Note that whenever $w'\in T_x^y$, then
    $w^{-1}w'\in T_x^x$. Now, $T_x^x$ is a finitely
    generated subgroup of $F_n$: it is the fundamental group
    of the graph whose vertices are the points in $X$ and
    (labelled) edges connect $x,y$ if $\varphi_i(x)=y$ for
    some $i\leq n$. Therefore, $T_x^y=wT_x^x$ is a coset of
    a finitely generated subgroup.
  \end{proof}

  \begin{claim}\label{rz}
    For every $m\in\N$ and $x_1,y_1,\ldots,x_m,y_m\in X$ the
    set $T_{x_1}^{y_1}\cdot\ldots\cdot T_{x_m}^{y_m}$ is
    closed in $F_n$ in the profinite topology.
  \end{claim}
  \begin{proof}
    This follows from the Ribes--Zalesski\v\i\ theorem
    \cite{ribes.zalesskii} and Claim \ref{cosets}.
  \end{proof}

  We need to define an extension of $X$. It will be obtained
  by dividing $X\times F_n$ by certain equivalence relation
  $\simeq$ so that $x\mapsto(x,e)\slash\simeq$ is an
  embedding. We will have to make sure that the extension is
  finite and define a metric on it so that the embedding is
  isometric. Before we make this definition precise, let us
  comment on how the metric on $(X\times F_n)\slash\simeq$
  will be defined. Note that there is a partial distance
  function $d_0$ on $X\times F_n$, namely for $(x,w),(y,w)$
  with $x,y\in X$ and $w\in F_n$ we put
  $d_0((x,w),(y,w))=d_X(x,y)$. Now, if $\simeq$ is an
  equivalence relation on $X\times F_n$, then there is a
  natural distance function on $(X\times F_n)\slash\simeq$
  defined as follows. If $C,D\in (X\times F_n)\slash\simeq$,
  then put $d_{(X\times F_n)\slash\simeq}(C,D)$ to be the
  minimum of $\Delta$ and the sums of the
  form 
  \begin{displaymath}\label{formula}
    \sum_{i=0}^{m-1} d_0(z_i,z_{i+1}')\tag{$**$}
  \end{displaymath}
  such that $z_0,z_1,z_1',\ldots,z_{m-1},z_{m-1}',z_m'\in
  X\times F_n$, the value $d_0(z_i,z_{i+1}')$ is defined for
  each $0\leq i<m$ and there is a sequence $C_0,\ldots,C_m$
  of elements of $X\times F_n\slash\simeq$ with
  $C_0=C,C_m=D$ and $z_0\in C_1$, $z_m'\in C_m$ and
  $z_j,z_j'\in C_j$ for $0<j<m$.

  Now we will define the equivalence relation $\simeq$ and
  check the details of the construction described above. For
  that, we need a couple of definitions. Given $x,y\in X$,
  a \textit{chain from $x$ to $y$} is a
  sequence $$z_0,z_1,z_1',\ldots,z_{m-1},z_{m-1}',z_m,z_m'\in
  X$$ such that $z_0=x,z_m=y$ and for each $1\leq i\leq m$
  there exists $w_i\in W_n$ such that $w_i(z_i)=z_i'$.  The
  \textit{distance of a chain}
  $z_0,z_1,z_1',\ldots,z_{m-1},z_{m-1}',z_m,z_m'$ is defined
  as $\sum_{i=0}^{m-1} d_A(z_i,z_{i+1}')$. A \textit{word
    realization} of a chain as above is a sequence of words
  $w_1,\ldots,w_m\in W_n$ such that $w_i(z_i)=z_i'$ for each
  $1\leq i\leq m$. Call a chain \textit{trivial} if it has a
  word realization $w_1,\ldots,w_m\in W_n$ such that
  $w_1\cdot\ldots\cdot w_m=e$ holds in $F_n$.

  Let now $M\in\N$ be such that $M\delta>\Delta$. Note that
  since $X$ is finite, there are only finitely many
  nontrivial chains $z_0,z_1,z_1',\ldots,z_m,z_m'\in X$ with
  $m\leq M$. For each nontrivial chain
  $z_0,z_1,z_1',\ldots,z_m,z_m'\in X$ with $m\leq M$, the
  set $T_{z_1}^{z_1'}\cdot\ldots\cdot T_{z_m}^{z_m'}$ is a
  closed subset of $F_n$ which does not contain $e$. Using
  Claim \ref{rz}, find a finite index normal subgroup $H\lhd
  F_n$ which is disjoint from every
  $T_{z_1}^{z_1'}\cdot\ldots\cdot T_{z_m}^{z_m'}$ as above.

  Write $Z$ for $X\times F_n$ and define an equivalence
  relation $\simeq$ on $Z$ as follows. Given $w_1,w_2\in
  F_n$ write $$(x_1,w_1)\simeq(x_2,w_2)$$ if there is $v\in
  F_n$ with $w_2^{-1}w_1H=vH$ and $v(x_1)=x_2$. Given
  $(x,w)\in Z$ write $[x,w]$ for its $\simeq$-class.

  Write $Y$ for $Z\slash\simeq$ and note that $Y$ is
  finite. The latter follows from the fact that if
  $F_n\slash H=\{d_1H,\ldots,d_tH\}$, then $Y=\{[x,d_i]:x\in
  X,i\leq t\}$. Now, define a metric $d_Y$ on $Y$ as
  follows. Let $d_Y([x,w],[y,v])$ be the minimum of $\Delta$
  and the set of sums of the form $$\sum_{i=0}^{m-1}
  d_X(z_i,z_{i+1}')$$ for sequences
  $z_0,z_1,z_1',\ldots,z_m,z_m'$ of elements of $X$ such
  that
  \begin{itemize}
  \item $z_0=x,z_m=y$,
  \item and there are $w_i\in F_n$ (for $0\leq i\leq m$)
    with $w_0=w,w_m=v$ and $(z_i',w_{i-1})\simeq(z_i,w_i)$
    for each $1\leq i\leq m$.
  \end{itemize}
  Note that a sum as above is equal to $0$ exactly when
  $z_i=z_{i+1}'$ for every $i<m$ and hence the definition of
  $d_Y$ does not depend on the representatives of
  $\simeq$-classes and defines a metric on $Y$. Note that
  this definition coincides with the formula given by
  (\ref{formula}).

  Define an embedding of $X$ into $Y$ via
  $x\mapsto[x,e]$. We claim that this is an isometric
  embedding and that each $\varphi_i$ extends to an isometry
  of $Y$. The second part is clear given the first one since
  for each $i\leq n$ the map $[x,w] \mapsto[x,a_iw]$ is
  well-defined and is easily seen to be an isometry of $Y$
  which extends $\varphi_i$. Thus, we only need to show that
  $x\mapsto[x,e]$ is an isometric embedding.

  \begin{claim}\label{disst}
    For any $x,y\in X$ and $w\in F_n$ the distance
    $d_Y([x,e],[y,w])$ is equal to the minimum of $\Delta$
    and the minimal distance of a chain from $x$ to $y$
    which has a word realization $v_1,\ldots,v_k$ such that
    $v_1\cdot\ldots\cdot v_kH=wH$.
  \end{claim}
  \begin{proof}
    If $z_0,z_1,z_1',\ldots,z_m,z_m'$ in $X$ are such that
    $$d_Y([x,e],[y,w])=\sum_{i=0}^{m-1} d_X(z_i,z_{i+1}')$$
    and $w_0,\ldots,w_n$ in $F_n$ are such that
    $w_0=e,w_m=w$ and $(z_i',w_{i-1})\simeq(z_i,w_i)$ for
    each $1\leq i\leq m$, then find $v_i\in F_n$ (for $1\leq
    i\leq m$) such that $v_iH=w_{i-1}^{-1}w_iH$ and
    $v_i(z_i')=z_i$. Then $z_0,z_1,z_1',\ldots,z_m,z_m'$ is
    a chain and $v_1,\ldots,v_m$ is its word realization
    with $v_1\ldots v_mH=w_0^{-1}w_1\ldots
    w_{m-1}^{-1}w_mH=w_mH=wH$.

    On the other hand, if $z_0,z_1,z_1',\ldots,z_m,z_m'$ in
    $X$ forms a chain from $x$ to $y$ with a word
    realization $v_1,\ldots,v_k$ such that
    $v_1\cdot\ldots\cdot v_kH=wH$, then write $w_0=e$ and
    $w_i=w_{i-1}v_i$ for $1\leq i<m$ and $w_m=w$.  Then the
    sequence $z_0,z_1,z_1'\ldots,z_m,z_m'$ together with
    $w_0,\ldots,w_m$ satisfies the conditions in the
    definition of $d_Y$.
  \end{proof}

  Consequently, as $d_X(x,y)\leq\Delta$, by Claim
  \ref{disst}, we have that $d_Y([x,e],[y,e])$ is the minimal
  distance of the chains from $x$ to $y$ which have a word
  realization $w_1,\ldots,w_k$ with $w_1\cdot\ldots\cdot
  w_k\in H$. Say that a chain $c$ from $x$ to $y$
  \textit{realizes the distance} if the distance of $c$ is
  equal to $d_Y([x,e],[y,e])$. We need to show that if a
  chain $c$ realizes the distance from $x$ to $y$, then its
  distance is equal to $d_X(x,y)$.

  \begin{claim}\label{cancellation}
    Suppose $x,y\in X$ and $c$ is a chain from $x$ to $y$
    with a word realization $w_1,\ldots,w_m\in W_n$. If
    $w_i=v_ia$ and $w_{i+1}=a^{-1}v_{i+1}$ for some $1\leq
    i< m$ with $v_i,v_{i+1}\in W_n$ and $a\in A$, then there
    is a chain $c'$ from $x$ to $y$ which has the same
    distance as $c$ and a word realization
    $w_1',\ldots,w_m'\in W_n$ such that $w_j'=w_j$ for
    $j\not=i,i+1$, $w_j=v_i$ for $j=i,i+1$.
  \end{claim}
  \begin{proof}
    Write $c=(z_0,z_1,z_1',\ldots,z_m,z_m')$. Let
    $\varphi=\varphi_k$ if $a=a_k$ and
    $\varphi=\varphi_k^{-1}$ if $a=a_k^{-1}$. Consider the
    chain $c'=(y_0,y_1,y_1',\ldots,y_m,y_m')$ with $y_j=z_j$
    for $j\not=i$ and $y_j'=z_j'$ for $j\not=i+1$, and
    $y_i=\varphi(z_i)$, $y_{i+1}'=v_{i+1}(z_{i+1})$. The
    distance of $c'$ is the same as that of $c$
    since $$d(y_i,y_{i+1}')=d(\varphi(z_i),v_{i+1}(z_{i+1}))=d(z_i,\varphi^{-1}(v_{i+1}(z_{i+1})))$$
    as $\varphi$ is an isometry. And we have
    $\varphi^{-1}(v_{i+1}(z_{i+1}))=w_{i+1}(z_{i+1})=z_{i+1}'$.
  \end{proof}

  Given two chains $c=(z_0,z_1,z_1'\ldots,z_m,z_m')$ and
  $c'=(z_0,z_1,z_1'\ldots,z_k,z_k')$, both from $x$ to $y$,
  say that $c$ is \textit{shorter than} $c'$ if $m<k$ and
  the distance of $c$ is not greater than that of $c'$.  Say
  that a word realization $w_1,\ldots,w_m$ of a chain
  \textit{has a trivial element} if there is $0<i<m$ with
  $w_i=e$.

  \begin{claim}\label{nontrivial}
    If a chain $c=(z_0,z_1,z_1'\ldots,z_m,z_m')$ from $x$ to
    $y$ has a word realization $w_1,\ldots,w_m$ with a
    trivial element, then there is a chain from $x$ to $y$
    which is shorter than $c$ and has a word realization
    $w_1',\ldots,w_k'$ with $w_1\ldots w_m=w_1'\ldots w_k'$.
  \end{claim}
  \begin{proof}
    Suppose $w_i=e$, i.e. $z_i=z_i'$. Consider the chain
    $$z_0,\ldots,z_{i-1},z_{i-1}',z_{i+1},z_{i+1}',\ldots,z_m,z_m'$$ and note that
    $w_1,\ldots,w_{i-1},w_{i+1},\ldots,w_m\in W_n$ is its
    word realization. The fact that this chain is shorter
    than $c$ follows from the triange inequality.
  \end{proof}

  \begin{claim}\label{different}
    If a chain $c=(z_0,z_1,z_1'\ldots,z_m,z_m')$ from $x$ to
    $y$ has a word realization $w_1,\ldots,w_m$ and
    $z_i=z_{i+1}'$ for some $0<i<m$, then there is a chain
    from $x$ to $y$ which is shorter than $c$ and has a word
    realization $w_1',\ldots,w_k'$ with $w_1\ldots
    w_m=w_1'\ldots w_k'$.
  \end{claim}
  \begin{proof}
    If $z_i=z_{i+1}'$, then consider the chain $c'$ of
    $y_0,y_1,y_1',\ldots,y_{m-1},y_{m-1}'$ with $y_j=z_j$
    for $j<i$, $y_j=z_{j+1}$ for $j\geq i$, $y_j'=z_j'$ for
    $j\leq i$ and $y_j'=z_{j+1}'$ for $j>i$. Note that it is
    still a chain from $x$ to $y$ with a word realization
    $w_1',\ldots,w_{m-1}'$ with $w_j'=w_j$ for $j<i$,
    $w_i'=w_iw_{i+1}$ and $w_j'=w_{j+1}$ for $j>i$.
  \end{proof}
  
  \begin{claim}\label{lengthone}
    If chain $c=(z_0,z_1,z_1'\ldots,z_m,z_m')$ from $x$ to
    $y$ realizes the distance from $x$ to $y$ and cannot be
    made shorter, then $m=1$ and $z_m=z_m'$.
  \end{claim}
  \begin{proof}
    Note that by Claim \ref{different} and the assumption
    that $M\delta>\Delta$ we have that $m\leq M$. First note
    that the chain must be trivial. Indeed, since otherwise,
    for any word realization $w_1,\ldots,w_m$ of $c$ we have
    $w_1\ldots w_m\in T_{z_1}^{z_1'}\cdot\ldots\cdot
    T_{z_m}^{z_m'}$ and the latter set is disjoint from $H$
    if the chain is nontrivial. Now, since the chain is
    trivial, it has a word realization $w_1,\ldots,w_m$ such
    that $w_1\ldots w_m=e$. Now, if $m\geq 2$, then Claims
    \ref{cancellation} and \ref{nontrivial} imply that the
    chain can be made shorter. Therefore, $m=1$ and $w_m=e$.
  \end{proof}
  
  Note finally that since $(x,y,y)$ is a chain from $x$ to
  $y$, Claim \ref{lengthone} implies that
  $d_Y([x,e],[y,e])=d_X(x,y)$ and we have that
  $x\mapsto[x,e]$ is an isometric embedding, as needed. This
  ends the proof.

\end{proof}

\begin{lemma}
  The Urysohn space has the extension property.
\end{lemma}
\begin{proof}
  This is a standard amalgamation argument. Note that since
  the language of metric spaces does not have any function
  symbols, instead of finitely generated structures, we talk
  about finite tuples. Suppose then that $\bar
  b=(b_1,\ldots,b_n),\bar c=(c_1,\ldots,c_m),\bar
  a=(a_1,\ldots,a_k)$ are finite tuples in $\U$. Write
  $B=\{b_1,\ldots,b_n\},C=\{c_1,\ldots,c_m\},A=\{
  a_1,\ldots,a_k\}$ and suppose $A\subseteq B\cap C$. Let
  $C'$ be copy of $C$ with $B\cap C'=A$ and let $D=B\cup C'$
  be a metric space with the metric $d_D$ such that
  $d_D\restriction B=d_\U\restriction B$, $d_D\restriction
  C'=d_\U\restriction C$ (under the natural identification)
  and if $b\in B, c\in C'$, then
  $d_D(b,c)=\min\{d_\U(b,a)+d_\U(a,c):a\in A\}$. Assume
  without loss of generality that $D$ is embedded into $\U$
  over $B$ and note that $C'\equiv_A C$ and $C'\ind_A
  B$. This ends the proof.
\end{proof}

To check that the Urysohn space admits isolated sequences,
we need to introduce a couple of definitions.  Given a
metric structure $M$ and a tuple $\bar a\in M^k$ and
$\varepsilon>0$ write $\ball_M(\bar a,\varepsilon)$ for
$\{x\in M: d_M(x,a_i)<\varepsilon\mbox{ for some }i\leq
k\}$. Suppose $M$ is a homogeneous metric structure, $\bar
a\in M^k$ for some $k\in\N$ and $p=\mtp(\bar a)$. We say
that a sequence $(\bar a_n:n\in\N)$ of elements of $p(M)$ is
\textit{isometrically isolated} if there exists a sequence
of $\varepsilon_n\in(0,\infty)$ and isometric
embeddings $$\eta_n: \ball_M(\bar a,\varepsilon_n)\to M$$
such that $\eta_n(\bar a)=\bar a_n$ and for every sequence
$\bar b_n\in\rng(\eta_n)$ such that $\mtp(\bar
b_n)=\mtp(\bar a_n)$ and $d_M(\bar a_n,\bar
b_n)<\varepsilon_n$ there is an automorphism $\varphi$ of
$M$ with $\varphi(\bar a_n)=\bar b_n$ for every $n\in\N$.

Note that any isometrically isolated sequence is isolated
since if $p=\mtp(\bar a)$ and $\eta: \ball_M(\bar
a,\varepsilon)\to M$ is an isometric embedding for some
$\varepsilon>0$, then $\{\bar b=(b_1,\ldots,b_k)\in p(M):
\forall i\leq k\quad b_i\in\rng(\eta_n)\}$ is relatively
$\varepsilon$-saturated over $\eta(\bar a)$.

Given $k\in\N$, say that a sequence $(\bar a_n:n\in\N)$ of
$k$-tuples of elements of a metric structure $M$ is
\textit{nontrivial convergent} if it is convergent as a
sequence in $M^k$ and if $\bar
a_\infty=(a^\infty_1,\ldots,a^\infty_k)$ is its limit and
$\bar a_n=(a^n_1,\ldots,a^n_k)$, then $a^n_i\not=
a^\infty_j$ and $a^n_i\not=a^m_j$ for any
$(n,i)\not=(m,j)\in\N^2$. In particular, $a_i^n\not=a_j^n$
for every $n\in\N$ and $i\not=j$. Note that, in case $k=1$,
a nontrivial convergent sequence is a convergent sequence
such that all its elements are distinct and different from
its limit.

A basic property of the Urysohn space that we will use in
the arguments below, due to Huhunai\v svili \cite{huhu}
(cf. \cite[Proposition 5.1.20]{pestov}), says that any
partial isometry between compact subspaces of $\U$ can be
extended to an isometry of $\U$.

\begin{lemma}\label{convergent}
  For every $k\in\N$ and a quantifier-free $k$-type $p$, any
  nontrivial convergent sequence in $p$ is isometrically
  isolated in $p$.
\end{lemma}
\begin{proof}
  Let $p$ be the quantifier-free type of $\bar
  a\in\U^k$. Write $\bar a=(a^1,\ldots,a^k)$ and let
  $\delta_{ij}=d_\U(a^i,a^j)$ for $i,j\leq k$.

  Let $\bar a_n$ be a nontrivial convergent sequence in
  $p$. Assume that $\bar a_n$ converges to $\bar
  a_\infty=(a^\infty_1,\ldots,a^\infty_k)$. Write $\bar
  a_n=(a^n_1,\ldots,a^n_k)$ for each $n\in\N$. For each
  $n,m\in\N$ and $i,j\leq k$ let
  $\delta^{nm}_{ij}=d_\U(a^n_i,a^m_j)$ and
  $\delta^{n\infty}_{ij}=d_\U(a^n_i,a^\infty_j)$ and note
  that $\lim_{m,n\to\infty}\delta^{nm}_{ij}=\delta_{ij}$ and
  $\lim_{n\to\infty}\delta^{n\infty}_{ij}=\delta_{ij}$.

  For each $n\in\N$ choose $\varepsilon_n>0$ such that
  $\varepsilon_n<\delta^{nm}_{ij}$ for each $m\not=n$ and
  $i,j\leq k$ as well as
  $\varepsilon_n<\delta^{nn}_{ij}=\delta_{ij}$ for all
  $i\not=j$, $i,j\leq k$. Such an $\varepsilon_n>0$ exists
  since the sequence $(\bar a_n:n\in\N)$ is nontrivial and
  $\lim_{m\to\infty}\delta^{nm}_{ij}=d_\U(a^n_i,a^\infty_j)>0$.

  For each $n\in\N$ write $A_n$ for
  $\ball(\{a^n_1,\ldots,a^n_k\},\varepsilon_n)$. Note that
  $A_n$ is a disjoint union of balls around the points
  $a^n_i$. Consider the metric space $B'$ which is the
  disjoint union $\bigcup_{n\in\N} B_n'$ with each $B_n'$ a
  copy of $A_n$ (say the copy of $a^n_i$ in $B_k$ is
  $a^n_i{'}$) and let the metric on $B'$ be defined so that
  it is equal to the original metric $d_\U$ on each $B_n'$
  and if $x,y\in\bigcup_{n\in\N} B_n'$ are such that $x\in
  B_n'$ and $y\in B_m'$ with $n\not=m$ and
  $x\in\ball(a^n_i{'},\varepsilon_n)$,
  $y\in\ball(a^m_j{'},\varepsilon_m)$,
  then $$d_{B'}(x,y)=\delta^{nm}_{ij}.$$ Note that $d_{B'}$
  is a metric by the choice of the numbers $\varepsilon_n$.
  Now let $B_\infty=\{a^\infty_1{'},\ldots,a^\infty_k{'}\}$
  be a copy of $\{a^\infty_1,\ldots,a^\infty_k\}$ and let
  $B=B'\cup B_\infty$ with the metric $d_{B}=d_{B'}$ on
  $B'$, $d_{B}=d_{B_\infty}$ on $B_\infty$ and if $x\in
  B'$ is such that $x\in B_n$ with
  $x\in\ball(a^n_j{'},\varepsilon_n)$ and $i\leq k$, then
  $$d_B(x,a^\infty_i{'})=\delta^{n\infty}_{ji}.$$ 
  Since the subspace of $B$ consisting of the points
  $a^n_i{'}$ and $a^\infty_i{'}$ for $n\in\N$ and $i\leq k$
  is compact and isometric to the subspace of $\U$
  consisting of the points $a^n_i$ and $a^\infty_i$ for
  $n\in\N$ and $i\leq k$, the Huhunai\v svili theorem
  \cite{huhu} implies that the map $a^n_i{'}\mapsto a^n_i$
  and $a^\infty_i{'}\mapsto a^\infty_i$ extends to an
  isometric embedding $\eta$ of $B$ into $\U$.

  For each $n\in\N$ write $B_n$ for the image of $B_n'$
  under $\eta$. Note that each $B_n$ is an isometric copy of
  $\ball(\bar a,\varepsilon_n)$. We claim that the sets
  $B_n$ (treated as the embeddings of $\ball(\bar
  a,\varepsilon_n)$), together with the numbers
  $\varepsilon_n$ witness that $\bar a_n$ is isometrically
  isolated. For that, pick a sequence of isometric
  embeddings $\varphi_n:\{a^n_1,\ldots,a_k^n\} \to B_n$ with
  $d_\U(\varphi_n(a^n_i),a^n_i)<\varepsilon_n$ for each
  $i\leq k$. Consider a partial isometry $\varphi'$ of $\U$
  with
  $\dom(\varphi')=\bigcup_{n\in\N}\{a^n_1\ldots,a_k^n\}\cup\{a^\infty_1,\ldots,a^\infty_k\}$
  such that $\varphi'(a^\infty_i)=a^\infty_i$ and
  $\varphi'(a^n_i)=\varphi_n(a^n_i)$. Note that $\varphi'$
  is a partial isometry of the Urysohn space with compact
  domain, so again by the Huhunai\v svili theorem
  \cite{huhu}, there is an isometry $\varphi\in\iso(\U)$
  that extends $\varphi'$. Clearly, $\varphi$ extends each
  $\varphi_n$, which shows that $B_n$ are as needed and the
  sequence is isometrically isolated. This ends the proof.
\end{proof}

\begin{proposition}
  The Urysohn space admits isolated sequences.
\end{proposition}
\begin{proof}
  Suppose $p$ is the quantifier-free $k$-type of a tuple
  $\bar a=(a_1,\ldots,a_k)$. First note that we can assume
  that $\bar a$ consists of distinct elements. Indeed,
  otherwise one can remove repetitions from $\bar a$ and
  work with a quantifier-free $m$-type $q$ for some
  $m<k$. Then, for every $m$-tuple $\bar b\in q(M)$ there is
  a unique tuple $\bar b'\in p(M)$, which contains $\bar b$
  such that
  \begin{itemize}
  \item if $(\bar b_n:n\in\N)$ is isolated in $q$, then
    $(\bar b_n':n\in\N)$ is isolated in $p$,
  \item the map $\bar b\mapsto \bar b'$ is a homeomorphism
    of $q(M)$ and $p(M)$.
  \end{itemize}

  Now, suppose $Z\subseteq p(\U)$ is nonmeager. Without loss
  of generality (restricting to an open subset of $p(\U)$ if
  neccessary), assume that $Z$ is nonmeager in every
  nonempty open set. Pick any $\bar a_\infty\in p(\U)$ with
  $\bar a_\infty=(a^\infty_1,\ldots,a^\infty_k)$ and note
  that $a^\infty_i\not=a^\infty_j$ for $i\not=j$. Using the
  assumption that $Z\cap V$ is nonmeager for every open
  neighborhood $V$ of $\bar a_\infty$, construct a sequence
  $\bar a_n$ of elements of $Z$ convergent to $\bar
  a_\infty$ such that if $\bar a_n=(a^n_1,\ldots,a_k^n)$,
  then $a^n_i\not= a^\infty_j$ and $a^n_i\not=a^m_j$ for any
  $n,m\in\N$ and $i,j\leq k$ with $(n,i)\not=(m,j)$. This
  sequence is then nontrivial convergent and hence isolated
  by Proposition \ref{convergent}.
\end{proof}

\section{The measure algebra}\label{sec:measure}

Recall that given a standard probability space $(X,\B,\mu)$
we define the equivalence $\approx$ on $\B$ by $A\approx B$
if $\mu(A\Delta B)=0$ and the measure algebra is the family
of equivalence classes of sets in $\B$. Given $A\in\B$ write
$[A]$ for its $\approx$-equivalence class (although we will
often abuse notation and write only $A$ instead of
$[A]$). The measure algebra is then the family of
$\approx$-classes of the sets in $\B$. It becomes a metric
space with the metric $d_\MALG([A],[B])=\mu(A\Delta B)$ and
we treat it as a metric structure together with this metric,
the operation of symmetric difference $\Delta$ and the empty
set as a constant. We write $\malg$ for the structure
$(\B\slash\!\!\approx,d_{\malg},\Delta,\emptyset)$.  The
Sikorski duality \cite[Theorem 15.9]{kechris} connects
automorphisms of $\malg$ and measure-preserving bijections
on the space $X$. In particular, it implies that the group
of automorphisms of $\malg$ with the topology of pointwise
convergence is isomorphic to the group of measure-preserving
bijections $\aut(X,\mu)$ with the weak topology (see
\cite[Section 1]{kechris.global}). For more details about
the measure algebra and the standard measure space we refer
the reader to \cite[Chapter 22]{handbook} and to
\cite[Chapter 32]{fremlin}.

Throughout the proofs below we often use the fact (cf
\cite[Lemma 7.10]{kechris.miller}) that whenever
$A,B\subseteq X$ have the same measure, then there is a
measure-preserving bijection $f:X\to X$ such that
$f(A)=B$. Below, given a finite subalgebra $\A$ of $\malg$,
we write $\at(\A)$ for the set of atoms of $\A$.

\begin{lemma}
  The measure algebra $\MALG$ has locally finite automorphisms.
\end{lemma}
\begin{proof}
  Note that finitely generated substructures of $\MALG$ are
  finite subalgebras. Thus, to show locally finite
  automorphism we need to prove the following. For every
  finite subalgebra $\A\subseteq\malg$ there exists a finite
  algebra $\B\subseteq\malg$ with $\A\subseteq\B$ such that
  every partial automorphism of $\A$ extends to an
  automorphism of $\B$. To see this, we need a couple of
  notions. Given finite $\A\subseteq\B\subseteq\malg$ and
  $A_1,A_2\in\A$ say that $A_1$ and $\A_2$ \textit{are
    identically partitioned by} $\B$ if the sets
  $\{\mu(A_1\cap B):B\in\at(\B),B\subseteq A_1\}$ and
  $\{\mu(A_2\cap B):B\in\at(\B),B\subseteq A_2\}$ (both
  counted with repetitions) are equal (up to a
  permutation). Note that if $\A\subseteq\B\subseteq\malg$
  are finite and such that every $A_1,A_2\in\A$ of the same
  measure are identically partitioned by $\B$, then any
  partial automorphism of $\A$ extends to an automorphism of
  $\B$. Moreover, it is enough to guarantee this for $A_1$
  and $A_2$ disjoint.

  Say that a finite extension $\A\subseteq\B$ is
  \textit{good} if every two atoms of $\A$ of the same
  measure are identically partitioned by $\B$. Note that
  this is a transitive relation and if $\A\subseteq\B$ is
  good and $A_1,A_2\in\A$ are identically partitioned by
  $\A$, then $A_1$ and $A_2$ are identically partitioned by
  $\B$.

  Enumerate as $((A_i,B_i):1\leq i< N)$ the set of all
  pairs $A,B$ of disjoint sets in $\A$ of the same
  measure. By induction on $i\leq N$ construct a sequence of
  finite subalgebras $\A_i\subseteq\malg$ with $\A_0=\A$
  such that for every $1\leq i\leq N$ we have
  \begin{itemize}
  \item $A_{i-1}\subseteq\A_i$ is good
  \item $A_i$ and $B_i$ are identically partitioned by
    $\A_{i+1}$.
  \end{itemize}
  After this is done, the algebra $\B=\A_N$ will be as
  needed. 

  It is enough to describe the induction step construction
  of $\A_{i+1}$ from $\A_i$.  Note first that (by shrinking
  $A_i$ and $B_i$ if neccessary) we can assume that for
  every $A,B\in\at(\A_i)$ with $A\subseteq A_i$ and
  $B\subseteq B_i$ we have $\mu(A)\not=\mu(B)$. Write
  $R=\{\mu(A): A\in\at(\A_i), A\subseteq A_i\}$ and
  $S=\{\mu(B): B\in\at(\A_i), B\subseteq B_i\}$, so that
  $R\cap S=\emptyset$. Write $a=\mu(A_i)=\mu(B_i)$ and let
  $U\subseteq X$ be a measurable set of measure $a$. Let
  $\C_1$ be the algebra of subsets of $A_i$ equal to
  $\A_i\restriction A_i$ and $\C_2$ be the algebra of
  subsets of $B_i$ equal to $\A_i\restriction B_i$. Find two
  algebras $\C_1'$ and $\C_2'$ of subsets of $U$ such that
  $\C_1'$ is isomorphic to $\C_1$, $\C_2'$ is isomorphic to
  $\C_2$ and $\C_1',\C_2'$ are (stochastically)
  independent. Fix measure-preserving bijections
  $\varphi:A_i\to U$, $\psi:B_i\to U$ such that $\varphi$
  maps $\C_1$ to $\C_1'$ and $\psi$ maps $\C_2$ to
  $\C_2'$. Write $\C'$ for the algebra of subsets of $U$
  generated by $\C_1'$ and $\C_2'$ and let $\C$ be the
  algebra of subsets of $A_i\cup B_i$ generated by the
  preimages $\varphi^{-1}(C)$ and $\psi^{-1}(C)$ for
  $C\in\C'$.  Note that the sets $A_i$ and $B_i$ are
  identically partitioned by the algebra generated by $\A_i$
  and $\C$. Note also that if $C\subseteq A_i$ is an atom of
  $\A_i$ of measure $r$, then $C$ is partitioned by $\C$
  into sets of measures $r\cdot\mu(D)$ for $D\in\at(\C_2)$
  and analogously, if $E\subseteq B_i$ is an atom of $\A_i$
  of measure $s$, then $D$ is partitioned by $\C$ into sets
  of measures $s\cdot\mu(F)$ for $F\in\at(\C_2)$. In order
  to construct a good extension of $\A_i$, partition every
  atom $A\in\at(\A_i)$ that is disjoint from $A_i\cup B_i$
  as follows:
  \begin{itemize}
  \item[(i)] if $\mu(A)\in R$, then partition $A$ into sets of
    measures $\mu(A)\cdot\mu(D)$, for $D\in\at(\C_2)$,
  \item[(ii)] if $\mu(A)\in S$, then partition $A$ into sets
    of measures $\mu(A)\cdot\mu(F)$, for $F\in\at(\C_1)$.
  \end{itemize}
  Let $\A_{i+1}$ be an extension of $\A_i$ generated by all
  the partitions as in (i) and (ii) above and by $\C$. Now
  $\A_{i+1}$ is a good extension of $\A_i$ and the sets
  $A_i$ and $B_i$ are identically partitioned by
  $\A_{i+1}$. This ends the construction and the proof.
\end{proof}

\begin{lemma}
  The measure algebra $\MALG$ has the extension property.
\end{lemma}
\begin{proof}
  Suppose $A,B,C$ are finitely generated subalgebras of
  $\MALG$ with $A\subseteq B\cap C$. Write $A_1,\ldots,A_n$
  for the set of atoms of $A$. Find an automorphism
  $\varphi$ of the measure space which fixes
  $A_1,\ldots,A_n$ and within each $A_i$ sends the atoms of
  $C$ contained in $A_i$ to sets which are (stochastically)
  independent from the atoms of $B$ contained in $A_i$. It
  is easy to see that $\varphi(C)\ind_A B$.
\end{proof}

To see that $\malg$ admits isolated sequences, we need to
understand which quantifier-free $\varepsilon$-types are
realized over finite tuples in $\malg$.

\begin{definition}
  Suppose $k\in\N$ and $\P=(A_1,\ldots,A_k)$ is a partition
  of $X$ into positive measure sets. Let $E=(e_{ij}:1\leq
  i,j\leq k)$ be a matrix of reals. Say that $E$ is
  \textit{$\P$-additive} if the following conditions hold:
  \begin{itemize}
  \item $e_{ii}\geq0$ and $0\leq e_{ij}\leq
    \mu(A_i)+\mu(A_j)$ for every $i,j\leq k$,
  \item the following equations are satisfied:
    \begin{eqnarray*}
      e_{ii}=\sum_{j\not=i}\mu(A_i)+\mu(A_j)-e_{ij}\\
      e_{ii}=\sum_{j\not=i}\mu(A_i)+\mu(A_j)-e_{ji}
    \end{eqnarray*}
  \end{itemize}

\end{definition}

\begin{claim}\label{measure.real}
  Suppose $\P=(A_1,\ldots,A_k)$ is a partition of $X$ into
  positive measure sets and $\varphi\in\aut(\malg)$. Let
  $e_{ij}=d_{\malg}(A_i,\varphi(A_j))$. Then the matrix
  $E=(e_{ij}:1\leq i,j\leq k)$ is $\P$-additive.
\end{claim}
\begin{proof}
  Let $f:X\to X$ be a measure-preserving bijection that
  induces $\varphi$. For $i\not=j$ write
  $\varepsilon_{ij}=\mu(f(A_i)\cap A_j)$. Note
  that $$e_{ij}=\mu(f(A_i)\Delta
  A_j)=\mu(A_j)+\mu(A_j)-2\varepsilon_{ij}.$$ This implies
  that $$e_{ii}=\mu(f(A_i)\Delta
  A_i)=2\sum_{j\not=i}\varepsilon_{ij}=\sum_{j\not=i}\mu(A_i)+\mu(A_j)-e_{ij}$$
  On the other hand, $$e_{ii}=\mu(f(X\setminus
  A_i)\Delta(X\setminus
  A_i))=2\sum_{j\not=i}\varepsilon_{ji}=\sum_{j\not=i}\mu(A_i)+\mu(A_j)-e_{ji}.$$
\end{proof}

\begin{lemma}\label{meas.sat}
  Let $\bar a=(A_1,\ldots,A_k)$ be a partition of $X$ into
  positive measure sets and let $p=\mtp(\bar a)$. Suppose
  $C_1,\ldots,C_k$ are such that $C_i\subseteq A_i$ for each
  $i\leq k$ and
  $\mu(C_1)=\ldots=\mu(C_k)>0$. Let $$M=\{(B_1,\ldots,
  B_k)\in p(\MALG): \forall i\not=j\leq k\ \ B_i\cap
  A_j\subseteq C_j\ \wedge\ A_i\setminus B_i\subseteq
  C_i\}.$$ Then $M$ is relatively $2\mu(C_1)$-saturated over
  $\bar a$.
\end{lemma}
\begin{proof}
  Write $\varepsilon=2\mu(C_1)$. Let
  $E_1,\ldots,E_k\in\malg$ be such that $\mu(E_i\Delta
  A_i)<\varepsilon$ and $\mtp(E_1,\ldots,E_k)=p$. Write
  $e_{ij}=\mu(E_i\Delta A_j)$ and note that
  $E=(e_{ij}:i,j\leq k)$ is $\bar a$-additive by Claim
  \ref{measure.real}. We need to find $(B_1,\ldots, B_k)\in
  M$ such that $\mu(B_i\Delta A_j)=e_{ij}$ for each $i,j\leq
  k$. For each $i\not=j$ write
  $$\varepsilon_{ij}=\frac{1}{2}(\mu(A_i)+\mu(A_j)-e_{ij})$$
  and note that by $\bar a$-additivity we have
  $$\sum_{j\not=i}\varepsilon_{ji}=\frac{1}{2}e_{ii}<\frac{1}{2}\varepsilon=\mu(C_i)$$
  Thus, we can find disjoint measurable sets
  $D_{ji}\subseteq C_i$ such that
  $\mu(D_{ji})=\varepsilon_{ji}$. Write
  $D_i=\bigcup_{j\not=i}D_{ji}$ and note that
  $\mu(D_i)=\frac{1}{2}e_{ii}$. Put $B_i=A_i\setminus
  D_i\cup\bigcup_{j\not=i}D_{ij}$ and note that since (by
  $\bar
  a$-additivity) $$\sum_{j\not=i}\varepsilon_{ji}=\sum_{j\not=i}\varepsilon_{ij}$$
  we have that $\mu(B_i)=\mu(A_i)$. The sets $B_i$ are
  pairwise disjoint since the sets $D_{ij}$ are disjoint and
  so $\mtp(B_1,\ldots,B_k)=\mtp(A_1,\ldots,A_k)$. Also, we
  have $$\mu(B_i\Delta
  A_j)=\mu(B_i)+\mu(A_j)-2\varepsilon_{ij}=e_{ij}$$ if
  $j\not=i$ and $$\mu(B_i\Delta
  A_i)=2\sum_{j\not=i}\varepsilon_{ij}=e_{ii}.$$ This shows
  that $(B_1,\ldots,B_k)$ is as needed and $M$ is relatively
  $\varepsilon$-saturated over $\bar a$.
\end{proof}

\begin{definition}
  Let $\bar a=(A_1,\ldots,A_k)$ be a tuple in $\MALG$ such
  that $\bar a$ is a partition of $X$ into positive measure
  sets. Write $p$ for $\mtp(\bar a)$.  Given a sequence
  $\bar a_n=(A^n_1,\ldots,A^n_k)$ in $p$ say that it is
  \textit{weakly independent} if there is a sequence
  $((C_1^n,\ldots,C_k^n):n\in\N)$ such that
  \begin{itemize}
  \item $C_i^n\subseteq A_i^n$ for each $i\leq k$ and
    $n\in\N$,
  \item $\mu(C_1^n)=\ldots=\mu(C_k^n)>0$ for each $n\in\N$,
  \item all sets $\{C_i^n:i\leq k,n\in\N\}$ are pairwise
    disjoint,
  \item if $m\not=n$, then $C_i^n\subseteq A^m_1$ for every
    $i\leq k$.
  \end{itemize}
\end{definition}

\begin{lemma}\label{stoch.indep}
  If a sequence $\bar a_n$ in $\MALG$ is weakly independent,
  then it is isolated.
\end{lemma}
\begin{proof}
  Let $k\in\N$ be such that each $\bar
  a_n=(A_1^n,\ldots,A_k^n)$ is an $k$-element partition of
  $X$. Suppose $((C_1^n,\ldots,C_k^n):n\in\N)$ witnesses
  that the sequence is weakly independent and let
  $\varepsilon_n=2\mu(C_1^n)$. Write $p$ for the
  quantifier-free type of $\bar a_n$ and let
  $$M_n=\{(B_1,\ldots, B_k)\in p(\MALG): \forall i\not=j\leq
  k\ \ B_i\cap A^n_j\subseteq C^n_j\ \wedge\  A^n_i\setminus
  B_i\subseteq C^n_i\}.$$

  We claim that the sequence of sets $M_n$ together with
  $\varepsilon_n$ witness that the sequence $\bar a_n$ is
  isolated. The fact that $M_n$ is relatively
  $\varepsilon_n$-saturated over $\bar a_n$ follows directly
  from Lemma \ref{meas.sat}.

  Suppose now that $\bar b_n=(B^n_1,\ldots,B^n_k)$ are such
  that each $\bar b_n$ belongs to $M_n$ and $\mtp(\bar
  b_n)=\mtp(\bar a_n)$. We need to find
  $\varphi\in\aut(\malg)$ such that $\varphi(\bar a_n)=\bar
  b_n$ for each $n\in\N$. For every $n\in\N$ and $i\not=j$
  let $D^n_{ij}=B^n_i\cap A^n_j\subseteq C^n_j$ and let
  $\varepsilon^n_{ij}=\mu(D^n_{ij})$. For each $i\leq k$ write
  $D^n_i=\bigcup_{j\not=i}D^n_{ji}$. Let
  $E^n_i=A^n_i\setminus B^n_i\subseteq C^n_i$ and note that
  $$\mu(E^n_i)=\sum_{j\not=i}\varepsilon^n_{ij}=\sum_{j\not=i}\varepsilon^n_{ji}$$ since
  $\mu(B^n_i)=\mu(A^n_i)$. For every $j\not=i$ find
  measurable sets $E^n_{ji}\subseteq E^n_i$ such that
  $E^n_i=\bigcup_{j\not=i}E^n_{ji}$ and
  $\mu(E^n_{ji})=\varepsilon^n_{ji}$. Now, for each $n\in\N$
  find a measure-preserving bijection $$f_n:\bigcup_{i\leq
    k}D^n_i\cup E^n_i\to\bigcup_{i\leq k}D^n_i\cup E^n_i$$
  such that
  \begin{itemize}
  \item $f_n(E^n_i)=D^n_i$ for each $i\leq k$,
  \item $f_n(D^n_{ij})=E^n_{ij}$ for every $i\not=j\leq k$
  \end{itemize}
  Let $f:X\to X$ be a measure-preserving bijection such that
  $f\supseteq f_n$ for each $n\in\N$ and $f$ is equal to the
  identity on the complement of the set
  $\bigcup_{n\in\N}\bigcup_{i\leq k}D^n_i\cup E^n_i$. Note
  that for $m\not=n$ and $i\leq k$, the set $D^m_i\cup
  E^m_i$ is contained in $C^m_i$, so the function $f_m$ maps
  $A^n_1$ into itself and the domain of $f_m$ is disoint
  from $A^m_i$ for $i>1$. This implies that $f(A^n_i)=B^n_i$
  and hence the autmomorphism of $\malg$ induced by $f$ is
  as needed. This ends the proof.
\end{proof}

\begin{proposition}
  The measure algebra $\MALG$ admits isolated sequences.
\end{proposition}
\begin{proof}
  Suppose $p$ is a quantifier-free $k$-type of a tuple $\bar
  a=(A_1,\ldots,A_k)$ in $\malg$. First note that we can
  assume that the elements of $\bar a$ form a partition of
  $X$ into positive measure sets. Otherwise, one can consider
  the atoms of the algebra generated by $\bar a$ and work
  with a quantifier-free $m$-type $q$ for $m$ equal to the
  number of these atoms. Then, for every $m$-tuple $\bar
  b\in q(M)$ there is a unique tuple $\bar b'\in p(M)$, such
  that the algebras generated by $\bar b$ and $\bar b'$ are
  the same and such that
  \begin{itemize}
  \item if $(\bar b_n:n\in\N)$ is isolated in $q$, then
    $(\bar b_n':n\in\N)$ is isolated in $p$,
  \item the map $\bar b\mapsto \bar b'$ is a homeomorphism
    of $q(M)$ and $p(M)$.
  \end{itemize}

  Now, suppose $Z\subseteq p(\malg)$ is nonmeager and assume
  (restricting to an open subset if neccessary) that $Z$ is
  nonmeager in every nonempty open set. Construct a sequence
  of $(A^n_1,\ldots,A^n_k)\in Z$ and positive measure
  pairwise disjoint sets $D^n_i\subseteq A^n_i$ (for $i\leq
  k$) together with positive reals $\delta^n_i$ such that
  for every $n\in\N$ we have
  \begin{itemize}
  \item[(i)] $D^n_1,\ldots, D^n_k\subseteq A^m_1$ if $m<n$,
  \item[(ii)] $\mu(D^n_i\cap A^{n+1}_1)=\delta^n_i$ for
    every $i\leq k$,
  \item[(iii)] $d_\malg(A^{n+l+1}_k,A^{n+l}_k)<\delta^n_k\slash
    2^{l+1}$ for every $l\geq 1$,
  \item[(iv)] $\mu(\bigcap_{m\leq n}A^m_1\setminus
    D^m_1)>0$.
  \end{itemize}
  After this is done, for every $i\leq k$ and $n\in\N$ put
  $C^n_i=D^n_i\cap\bigcap_{m>n}A^m_1$ and note that by (ii)
  and (iii) above we have that $\mu(C^n_i)>0$ for each
  $i\leq k$ and $n\in\N$. By shrinking the sets $C^n_i$ if
  neccessary, we can assume that
  $\mu(C^n_1)=\ldots=\mu(C^n_k)$. Then the definition of
  $C^n_i$ and condition (i) above imply that if $m\not=n$,
  then $C^n_i\subseteq A^m_1$. Given that $D^n_i$ are
  pariwise disjoint, so are the sets $C^n_i$ and so the
  sequence $(A^n_1,\ldots,A^n_k)$ is weakly independent, as
  witnessed by $C^n_i$ and hence isolated by Lemma
  \ref{stoch.indep}.

  To perform the induction step, suppose we have constructed
  $(A^j_1,\ldots,A^j_k)\in Z$ for $j\leq n$, the sets
  $D^j_i$ for $i\leq n$ and $i\leq k$ as well as
  $\delta^j_i$ for $j\leq n-1$ and $i\leq k$. Consider the
  open set
  \begin{displaymath}
    U=\{(A_1,\ldots,A_k)\in p(\malg): \forall i\leq k\ \
    d_\malg(A_i,A^n_i)\leq\min_{m<n}
    \frac{\delta^m_i}{2^{n-m+1}}\}.  
  \end{displaymath}
  Write also $F=\bigcap_{m\leq n}A^m_1\setminus D^m_1$ and
  note that by the inductive assumption (iv), we have
  $\mu(F)>0$. Using the fact that for any positive measure
  set $E$ and $i\leq k$ the set $\{(A_1,\ldots,A_n)\in
  p(\malg):\mu(A_i\cap E)=0\}$ is closed nowhere dense, find
  $(A^{n+1}_1,\ldots,A^{n+1}_k)\in U\cap Z$ such that
  \begin{itemize}
  \item[(a)] $\mu(F\cap A^{n+1}_i)>0$ for every $i\leq k$,
  \item[(b)] $\mu(A^{n+1}_1\cap D^n_i)>0$ for every $i\leq k$.
  \end{itemize}
  Now, using (a) above, find $D^{n+1}_i\subseteq
  A^{n+1}_i\cap F$ such
  that $$0<\mu(D^{n+1}_i)<\frac{1}{2}\mu(A^{n+1}_i\cap F).$$
  This implies that the inductive condition (iv) will be
  satisfied at the next step. Put
  $\delta^n_i=\mu(A^{n+1}_1\cap D^n_i)$ and note that
  $\delta^n_i>0$ by (b) above. Condition (iii) holds because
  $(A^{n+1}_1,\ldots,A^{n+1}_k)\in U$ and thus, this
  concludes the induction step. This ends the proof.
\end{proof}

\section{The Hilbert space}\label{sec:hilbert}

The orthogonal group $O(\ell_2)$ is the group of
automorphism of the (real) Hilbert space. The Hilbert space
here is treated as the metric structure with the first sort
being $(\ell_2,0,+)$ and the second sort being the real line
with the field structure (including the inverse function
defined on non-zero elements by $x\mapsto x^{-1}$ and
mapping $0$ to $0$, as well as the function $x\mapsto -x$)
and constants for the rationals. We also add to the language
the multiplication by scalars function $\cdot:\R\times
\ell_2\to\ell_2$ (i.e. $(a,v)\mapsto a\cdot v)$ as well as
the inner product function
$\langle\cdot,\cdot\rangle:\ell_2\times\ell_2\to\R$.

Recall that by the Mazur--Ulam theorem \cite{mazur.ulam} any
isometry of a normed vector space which preserves zero, is a
linear isomorphism (in case of the Hilbert space this is
even simpler than the general case of the Mazur--Ulam
theorem), so we could also consider the structure only with
the constant $0$ and the inner product function. Still
another way would be to look at the unit sphere in the
Hilbert space equipped only with the metric (as a metric
space with no additional structure) and then the orthogonal
group would be the group of isometries of the sphere. We
will however, use the above language, as it seems the most
natural, and we will make use of it in order to talk about
substructures of the Hilbert space.

The unitary group $U(\ell_2)$ is the automorphism group of
the complex Hilbert space and the arguments below apply in
the same way to the complex Hilbert space, so we will focus
only of the real Hilbert space.

\begin{claim}
  If $A$ is a finitely generated substructure of the Hilbert
  space, then there exists a countable field $\K\subseteq\R$
  such that $\Q\subseteq\K$ and $A$ is a finite-dimensional
  $\K$-vector space
\end{claim}
\begin{proof}
  Let $\K$ consist of the elements of $A$ which are of the
  second sort. Since the language contains constants for the
  rationals, we have $\Q\subseteq\K$ and since the language
  contains the language of fields, $\K$ is a field. Clearly
  then $A$ is a $\K$-vector space and the dimension is
  bounded by the number of generators of $A$.
\end{proof}

\begin{lemma}\label{hilbert:loc.fin.aut}
  The Hilbert space $\ell_2$ has locally finite
  automorphisms.
\end{lemma}
\begin{proof}
  In fact, $\ell_2$ has the following stronger property. For
  any finitely generated substructure $A\subseteq\ell_2$,
  any isomoprhism between finitely generated substructures
  of $A$ extends to an automorphism of $A$. To see this, let
  $A_1,A_2\subseteq A$ be finitely generated substructures
  and $\varphi:A_1\to A_2$ be an isomorphism. Let $\K$ and
  $\K_1,\K_2\subseteq \K$ be such that $A$ is a $\K$-vector
  space, $A_1$ is $\K_1$-vector space and $A_2$ is a
  $\K_2$-vector space. Write $A_1'$ for the $\K$-vector
  space generated by $A_1$ and $A_2'$ for the $\K$-vector
  space generated by $A_2$. Note that since $\varphi$
  preserves the inner product, it is an isometry and since
  both $\K_1$ and $\K_2$ contain $\Q$, the map $\varphi$ can
  be extended to an isomorphism $\varphi':A_1'\to
  A_2'$. Now, since $\K$ is a field, the usual Gram--Schmidt
  orthogonalization process gives orthogonal bases
  $\{b^1_1,\ldots,b^1_k\}$ and $\{b^2_1,\ldots,b^2_k\}$ for
  the orthogonal complements of $A_1'$ in $A$ and $A_2$ in
  $A$ (respectively). The map which extends $\varphi$ and
  maps $b^1_i$ to $b^2_i$ extends to an automorphism of $A$.
\end{proof}

Note that the above proof also shows that given a finitely
generated substructure $A$ of the Hilbert space and its
finitely generated substructure $C\subseteq A$, we can form
the orthogonal complement $A\ominus C$ inside $A$ using the
standard Gram--Schmidt process. The extension property for the
Hilbert space is then straightforward and based on the
following claim.

\begin{claim}\label{indep.hilbert}
  Given finitely generated substructures $A,B,C\subseteq \ell_2$
  with $C\subseteq A\cap B$, if $A\ominus C\perp B\ominus
  C$, then $A\ind_C B$.
\end{claim}
\begin{proof}
  This is elementary linear algebra and the proof is
  analogous to that of Lemma \ref{hilbert:loc.fin.aut}.
\end{proof}

\begin{corollary}
  The Hilbert space $\ell_2$ has the extension property.
\end{corollary}
\begin{proof}
  Given finite-dimensional subspaces $A,B,C\subseteq \ell_2$
  with $C\subseteq A\cap B$ find a copy $D\subseteq\ell_2$
  of $B\ominus C$ which is orthogonal to $A$. Then $C\oplus
  D$ witnesses the extension property by Claim
  \ref{indep.hilbert}.
\end{proof}

Before we show that the Hilbert space admits isolated
sequences, we need a couple of lemmas. Below, given a closed
subspace $V\subseteq \ell_2$ and a vector $v\in\ell_2$ write
$\pi_V(v)$ for the projection of $v$ onto $V$. Also,
$\ball_{\ell_2}(v,\varepsilon)$ stands for the open ball
$\{w\in\ell_2: ||w-v||<\varepsilon\}$ and
$\sphere_{\ell_2}(v,\varepsilon)$ stands for the sphere
$\{w\in\ell_2: ||w-v||=\varepsilon\}$. Recall also that
$\bar v=(v_1,\ldots,v_k)\in\ell_2$ is an \textit{orthonormal
  tuple} if $||v_i||=1$ and $v_i\perp v_j$ for $i\not=j$

\begin{lemma}\label{epsilon}
  Suppose $\bar v=(v_1,\ldots,v_k)$ is an orthonormal tuple
  in $\ell_2$ and let $H\subseteq\ell_2$ be an
  infinite-dimensional closed subspace. Suppose
  $V_1,\ldots,V_k\subseteq\ell_2$ are closed
  infinite-dimensional subspaces such that $v_i\in V_i$ and
  $V_i\perp V_j$ for $i\not= j$. Write $H_i=V_i\cap H$ and
  suppose $H_i$ is infinite-dimensional and that
  $\pi_{H_i}(v_i)\not=0$ for each $i\leq k$. Then there
  exists $\varepsilon>0$ such that for every $\bar
  v'=(v_1',\ldots,v_k')$ such that $$\bar v'\equiv\bar
  v,\quad v_i'\in V_i\quad \mbox{and}\quad
  ||v_i'-v_i||<\varepsilon$$ for each $i\leq k$, there
  exists $\bar v''=(v_1'',\ldots,v_k'')$ such that $$\bar
  v''\equiv_{\bar v}\bar v'\quad\mbox{and}\quad v_i''-v_i\in
  H_i$$ for each $i\leq k$.
\end{lemma}
\begin{proof}
  Note that since the subspaces $V_i$ are mutually
  orthogonal, it is enough to prove the lemma for
  $k=1$. Assume then $V_1=\ell_2$ and write $v=v_1$ so that
  $\pi_H(v)\not=0$ (i.e. $v\not\perp H$). We need to show
  that there exists $\varepsilon>0$ such that for every
  $v'\in\sphere_{\ell_2}(0,1)$ with $||v'-v||<\varepsilon$
  there exists $v''\in\sphere_{\ell_2}(0,1)$ with $v''-v\in
  H$ and $v''\equiv_v v'$. The latter is equivalent to
  $||v''-v||=||v'-v||$ (since $v',v''$ have the same
  norm). Since $v\not\perp H$, there exists
  $w\in\sphere_{\ell_2}(0,1)$ such that $w\not=v$ and
  $w-v\in H$. Let $\varepsilon=||w-v||$. Write
  $S=\sphere_{\ell_2}(0,1)\cap(H+v)$ and note that
  $S=\sphere_{\ell_2}(0,1)\cap A$ for some
  infinite-dimensional closed affine subspace $A$ of
  $\ell_2$. Hence, $S$ is homeomorphic to the sphere
  $\sphere_{\ell_2}(0,1)$ and thus is connected. By the
  intermediate-value theorem, the function $f:S\to\R$ given
  by $f(s)=||s-v||$ assumes all values between $0$ and
  $\varepsilon$ on $S$, and so for every
  $v'\in\sphere_{\ell_2}(0,1)$ with $||v'-v||<\varepsilon$
  there exists $v''\in S$ with $||v''-v||=||v'-v||$. This
  ends the proof.
\end{proof}

\begin{lemma}\label{generates}
  Suppose $\bar v=(v_1,\ldots,v_k)$ is an orthonormal tuple
  in $\ell_2$ and $V_1,\ldots,V_k\subseteq\ell_2$ are closed
  infinite-dimensional subspaces such that $v_i\in V_i$ and
  $V_i\perp V_j$ for $i\not= j$. Write $$T=\{\bar
  w=(w_1,\ldots,w_k): \bar w \equiv\bar v\ \wedge\ \forall
  i\leq k\ w_i\in V_i\}.$$ Then $T$
  $(2,\varepsilon)$-generates an open set, for every
  $\varepsilon>0$.
\end{lemma}
\begin{proof}
  Fix $\varepsilon>0$. Find $\bar v'=(v_1',\ldots,v_k')$ in $\ell_2$
  such that 
  \begin{itemize}
  \item $\bar v'\equiv\bar v$
  \item $v_i'\perp v_j$ for every $i\not=j$
  \item for every $i\leq k$ we have $\pi_{V_j}(v_i)\not=0$
    for every $j\not=i$.
  \end{itemize}
  For each $i,j\leq k$ write $v_{ij}'$ for $\pi_{V_j}(v_i')$
  and note that if $i\not=j$, then $v_{ij}'\perp v_j$.

  Find $\delta>0$ such that for every $i\leq k$ the
  following holds: for every sequence $(v''_j:j\not=i,j\leq
  k)$ of vectors in $V_i$ such that
  $||v''_j-v_{ji}'||<\delta$ there exists $\tilde v\in V_i$
  with $||\tilde v||=1$, $\tilde v\perp v''_j$ for every
  $j\not=i$ and $||\tilde
  v-v_i||<\varepsilon\slash2$. Assume without loss of
  generality that
  $\delta<\varepsilon\slash2$. Write $$U=\{\bar
  v''=(v_1'',\ldots,v_k''): \bar v''\equiv \bar v\ \wedge\
  d_{\ell_2}(\bar v'',\bar v')<\delta\}.$$
  \begin{claim}\label{hilbertclaim}
    For every $\bar v''\in U$ there are
    $\varphi_1,\varphi_2\in O(\ell_2)$ such that
    $$\varphi_2\varphi_1(\bar v)=\bar v''$$ and
    $\varphi_1(\bar v),\varphi_2(\bar v)\in T$, as well as
    $d_{\ell_2}(\varphi_1(\bar v),\bar v)<\varepsilon$ and
    $d_{\ell_2}(\varphi_2(\bar v),\bar v)<\varepsilon$.
  \end{claim}
  \begin{proof}
    Fix $\bar v''$ in $U$. Note that, by the choice of
    $\delta$, for each $i\leq k$ there exists $\tilde v_i\in
    V_i$ such that $||\tilde v_i||=1$, $||\tilde
    v_i-v_i||<\varepsilon\slash2$ and $\tilde v_i\perp
    \pi_{V_i}(v''_j)$ for every $j\not=i$. Now, for every
    $i\leq k$ find $w_i\in V_i$ such that
    \begin{equation}
      \label{eq:hilbertclaim}
      \mtp(w_i\slash v_i)=\mtp(v_i''\slash\tilde v_i).\tag{$\dagger$}
    \end{equation}
    Such vectors $w_i$ exist since each $V_i$ is isomorphic
    to $\ell_2$. Now, (\ref{eq:hilbertclaim}) implies
    that $$\mtp(w_i v_i)=\mtp(v_i'' \tilde v_i)$$ for each
    $i\leq k$ and hence the map $\psi_i$ such that
    $$\psi_i:w_i\mapsto v_i'',\quad\psi_i:v_i\mapsto\tilde
    v_i$$ is a partial automorphism of $\ell_2$ for each
    $i\leq k$. Now, since for $i\not= j$ both the domains
    and ranges of $\psi_i$ and $\psi_j$ are pairwise
    orthogonal, the map $\bigcup_{i\leq k}\psi_i$ is a
    partial automorphism of $\ell_2$. Extend $\bigcup_{i\leq
      k}\psi_i$ to $\varphi_2\in O(\ell_2)$.  Find also
    $\varphi_1\in O(\ell_2)$ such
    that $$\varphi_1:v_i\mapsto w_i$$ for each $i\leq k$.

    Note that since $||\tilde v_i- v_i||<\varepsilon\slash2$
    and $||v_i''-v_i||<\varepsilon\slash2$, we have
    $||v_i''- \tilde v_i||<\varepsilon$ and hence
    (\ref{eq:hilbertclaim}) implies that $||w_i-
    v_i||<\varepsilon$ for each $i\leq k$. Therefore,
    $d_{\ell_2}(\varphi_1(\bar v),\bar v)<\varepsilon$. Also
    $d_{\ell_2}(\varphi_2(\bar v),\bar v)<\varepsilon$, as
    well as $\varphi_1(\bar v)\in T$ and $\varphi_2(\bar
    v)\in T$. As we clearly have $\varphi_2\varphi_1(\bar
    v)=\bar v''$, this proves the claim.
  \end{proof}
  Claim \ref{hilbertclaim} clearly means that $T$
  $(2,\varepsilon)$-generates an open set, so this ends the
  proof.
\end{proof}

\begin{lemma}\label{weak.cor}
  Suppose $\bar v=(v_1,\ldots,v_k)$ is an orthonormal tuple
  in $\ell_2$ and let $H\subseteq\ell_2$ be an
  infinite-dimensional closed subspace such that the vectors
  $\pi_H(v_1),\ldots,\pi_H(v_k)$ are linearly
  independent. Write $$N=\{\bar w=(w_1,\ldots,w_k): \bar
  w\equiv\bar v\ \wedge\ \forall i\leq k\ w_i-v_i\in H\}.$$
  Then there exists $\varepsilon>0$ such that $N$ is
  $2$-relatively $\varepsilon$-saturated over $\bar v$.
\end{lemma}
\begin{proof}
  Write $w_j=\pi_H(v_i)$ for each $i\leq k$.
  \begin{claim}\label{claimoberwolfach}
     There exist $w_1',\ldots,w_k'\in H$ such that
     $w_i'\perp w_j'$ and $w_i'\perp w_j$ for $i\not=j\leq
     k$ and $w_i'\not\perp w_i$ for every $i\leq k$.
  \end{claim}
  \begin{proof}
    Inductively on $i\leq k$ construct $w_i'\in H$ such that
    $w_i'\not\perp w_i$ and $w_i'\perp w_j$ for $j\not=i$ and $w_i'\perp
    w_j'$ for $j<i$ as well
    as $$w_1,w_1',\ldots,w_i,w_i',w_{i+1},\ldots,w_k$$ are
    linearly independent. Suppose $w_1',\ldots,w_{i-1}'$ have
    been
    constructed. Let $$W_i=\{w_1,w_1',\ldots,w_{i-1},w_{i-1}',w_{i+1},\ldots,w_k\}^\perp\cap H$$
    and note that since
    $w_i\notin\lspan(w_1,w_1',\ldots,w_{i-1},w_{i-1}',w_{i+1},\ldots,w_k)$,
    we have that $W_i'=W_i\cap\{w_i\}^\perp$ is a proper
    subspace of $W_i$. Also,
    $W_i''=W_i\cap\lspan\{w_1,w_1',\ldots,w_i,w_i',w_{i+1},\ldots,w_k\}$
    is a proper subspace of $W_i$ since $W_i$ is
    infinite-dimensional. Now, $W_i'\cup W_i''$ do not cover
    $W_i$, so find $w_i'\in W_i\setminus(W_i'\cup W_i'')$
    and note that it is as needed.
  \end{proof}

  Using Claim \ref{claimoberwolfach}, find closed
  infinite-dimensional subspaces $V_i$ for $i\leq k$ such
  that for each $i\not= j\leq k$ we have
  \begin{itemize}
  \item $v_i\in V_i$ and $V_i\perp V_j$,
  \item $H\cap V_i$ is infinite-dimensional,
  \item $\pi_{H\cap V_i}(v_i)\not=0$.
  \end{itemize}

  Find $\varepsilon>0$ as in Lemma \ref{epsilon} and
  let $$T=\{\bar w=(w_1,\ldots,w_k): \bar w \equiv\bar v\
  \wedge\ \forall i\leq k\ w_i\in V_i\}.$$ Then $N$ is
  $T$-relatively $\varepsilon$-saturated by Lemma
  \ref{epsilon} and $T$ $(2,\varepsilon)$-generates an open
  set, by Lemma \ref{generates}. This ends the proof.
\end{proof}

\begin{definition}
  Say that a sequence of $k$-tuples $\bar a_n$ in $\ell_2$
  is \textit{strongly linearly independent} if there is a
  sequence of infinite-dimensional closed subspaces
  $V_n\subseteq\ell_2$ such that
  \begin{itemize}
  \item $V_n\perp V_m$ for $n\not= m$,
  \item $\bar a_m\perp V_n$ for $n\not=m$,
  \item the projections of the elements of $\bar a_n$ to
    $V_n$ are linearly independent.
  \end{itemize}
\end{definition}

\begin{lemma}\label{strongly.linearly}
  If $p$ is a quantifier-free type of an orthonormal tuple
  in $\ell_2$, then any strongly linearly independent
  sequence in $p(\ell_2)$ is $2$-weakly isolated.
\end{lemma}
\begin{proof}
  Suppose $\bar a_n=(a^n_1,\ldots,a^n_k)$ is strongly
  linearly independent in $p$. Note that
  $a_1^n,\ldots,a^n_k$ form an orthonormal tuple. Let
  $$N_n=\{\bar v=(v_1,\ldots,v_k)\in p(\ell_2):\forall i\leq
  k\  v_i-a^n_i\in V_n\}.$$

  We claim that there are $\varepsilon_n>0$ such that the
  sequence of $N_n$ and $\varepsilon_n$ witnesses that $\bar
  a_n$ is $2$-weakly isolated. For each $n$ find
  $\varepsilon_n>0$ as in Lemma \ref{weak.cor} for $\bar
  v=\bar a_n$. Then $N_n$ is $2$-relatively
  $\varepsilon_n$-saturated over $\bar a_n$.

  Suppose now that $\bar b_n=(b^n_1,\ldots,b^n_k)\in N_n$
  are such that $\mtp(\bar b_n)=\mtp(\bar a_n)$ for each
  $n\in\N$ and $d_{\ell_2}(\bar b_n,\bar
  a_n)<\varepsilon_n$. Then $b^n_i-a^n_i\in V_n$ for each
  $i\leq k$. Find $\varphi_n\in O(V_n)$ such that
  $\varphi_n(\pi_{V_n}(\bar a_n))=\pi_{V_n}(\bar b_n)$ and
  let $\varphi\in O(\ell_2)$ be such that $\varphi$ extends
  all the $\varphi_n$ and is equal to the identity on the
  orthogonal complement of the union of $V_n$'s. Then
  $\varphi(\bar a_n)=\bar b_n$ for each $n\in\N$. This ends
  the proof.
\end{proof}

\begin{proposition}
  The Hilbert space $\ell_2$ admits $2$-weakly isolated
  sequences.
\end{proposition}
\begin{proof}
  Suppose $p$ is a quantifier-free $k$-type of a tuple $\bar
  a=(a_1,\ldots,a_k)$ in $\ell_2$. First note that we can
  assume that the elements of $\bar a$ form an orthonormal
  set. Otherwise, one can consider a tuple which is an
  orthonormal basis for the space spanned by $\bar a$ and
  work with a quantifier-free $m$-type $q$ for some
  $m\leq n$. Then, for every $m$-tuple $\bar b\in q(M)$ there is
  a unique tuple $\bar b'\in p(M)$ such that the linear
  spans of $\bar b$ and $\bar b'$ are the same and
  \begin{itemize}
  \item if $(\bar b_n:n\in\N)$ is isolated in $q$, then
    $(\bar b_n':n\in\N)$ is isolated in $p$,
  \item the map $\bar b\mapsto \bar b'$ is a homeomorphism
    of $q(M)$ and $p(M)$.
  \end{itemize}
  In fact, for simplicity of notation, assume that $k=1$
  (the argument for arbitrary $k$ is analogous).

  Suppose now that $Z\subseteq p(\ell_2)$ is nonmeager.
  Restricting to an open subset of $p(\ell_2)$ if
  neccessary, we can assume that $Z$ is nonmeager in every
  nonempty open subset of $\ell_2$.

  Write $\gr(\ell_2)$ for the space of all closed subspaces
  of $\ell_2$ and $\gr(\ell_2,\infty)$ for the space of
  infinite-dimensional closed subspaces of $\ell_2$. The
  topology on $\gr(\ell_2)$ is induced from the strong
  operator topology via the map $V\mapsto\pi_V$. Write
  $d_\gr$ for a compatible metric on $\gr(\ell_2)$. Note
  that there is a sequence of functions
  $\rho_n:\gr(\ell_2,\infty)\to(0,\infty)$ such that
  whenever $W_n\in\gr(\ell_2,\infty)$ is a decreasing
  sequence of infinite-dimensional closed subspaces of
  $\ell_2$ and $d_\gr(W_n,W_{n+1})<\rho_{n+1}(W_n)$, then
  $\bigcap_n W_n$ is also infinite-dimensional.

  By induction on $n\in\N$ find vectors $a_n\in Z$, positive
  reals $\delta_n$ and pairwise orthogonal
  infinite-dimensional closed subspaces $W_n\subseteq
  \ell_2$ such that:
  \begin{itemize}
  \item[(i)] $W_1\oplus\ldots\oplus W_n$ is co-infinite
    dimensional,
  \item[(ii)] $\pi_{W_n}(a_n)\not=0$.
  \item[(iii)]
    $||\pi_{W_n\cap\,\lspan(a_{n+1})^\perp}(a_n)||=\delta_n$,
  \item[(iv)] if $m<n$, then $a_m\perp W_n$
  \item[(v)] if $m<n$, then we
    have $$||\pi_{W_m\cap\,(\bigcup_{i=m+1}^n
      \lspan(a_i))^\perp}(a_m)||>\frac{1}{2}\delta_m,$$
  \item[(vi)] if $m<n$ and
    $\varepsilon=\rho_n(W_m\cap(\bigcup_{i=m+1}^{n-1}
    \lspan(a_i))^\perp)$,
    then $$d_\gr(W_m\cap(\bigcup_{i=m+1}^n
    \lspan(a_i))^\perp,W_m\cap(\bigcup_{i=m+1}^{n-1}
    \lspan(a_i))^\perp)<\varepsilon.$$
  \end{itemize}
  After this is done, put
  $V_n=W_n\cap(\bigcup_{m>n}\lspan(a_m))^\perp$. Note that
  $V_n$ are infinite-dimensional by (vi) and mutually
  orthogonal given that $W_n$ are mutually orthogonal. Also,
  (iv) and the definition of $V_n$ imply that if $n\not=m$,
  then $a_m\perp V_n$. The projection of $a_n$ onto $V_n$ is
  nonzero by the condition (v) and hence $\bar a_n$ is
  strongly linearly independent, as witnessed by $V_n$ and
  hence $2$-weakly isolated by Lemma
  \ref{strongly.linearly}.

  To perform the induction step, suppose $a_1,\ldots, a_n$
  and $W_1,\ldots,W_n$ as well as
  $\delta_1,\ldots,\delta_{n-1}$ are chosen. Using the fact
  that a proper subspace of $\ell_2$ is meager as well as
  the assumption that $Z$ is nonmeager in any nonempty open
  set, find $a_{n+1}\in Z$ which does not belong to
  $\lspan(\bigcup_{i=1}^{n} W_i\cup\{a_i\})$ and
  \begin{equation}
    \label{eq:projection}
    a_{n+1}\not\perp W_n\cap\lspan(\pi_{W_n}(a_n))^\perp\tag{$\dagger\dagger$}
  \end{equation}
  and $a_{n+1}$ is so close to $a_n$ that for $m<n$ we
  have $$||\pi_{W_m\cap\,(\bigcup_{i=m+1}^n
    \lspan(a_i))^\perp}(a_m)||>\frac{1}{2}\delta_m$$ and for
  every $m<n$, writing
  $\varepsilon_m^n=\rho_{n+1}(W_m\cap(\bigcup_{i=m+1}^n
  \lspan(a_i))^\perp)$ we
  have $$d_\gr(W_m\cap(\bigcup_{i=m+1}^{n+1}
  \lspan(a_i))^\perp,W_m\cap(\bigcup_{i=m+1}^n
  \lspan(a_i))^\perp)<\varepsilon_m^n.$$ This implies that
  (v) and (vi) are satisfied at the induction step.

  Note that the projection of $a_{n+1}$ to
  $(\lspan(\bigcup_{i=1}^n W_i\cup\{a_i\}))^\perp$ is
  nonzero. Find an infinite-dimensional closed space
  $W_{n+1}$ such that
  \begin{itemize}
  \item $W_{n+1}$ is orthogonal to $\lspan(\bigcup_{i=1}^n
    W_i\cup\{a_i\})$,
  \item the projection of $a_{n+1}$ onto $W_{n+1}$ is
    nonzero,
  \item $W_1\oplus\ldots\oplus W_{n+1}$ is co-infinite
    dimensional.
  \end{itemize}
  This gives (i), (ii) and (iv). Finally, we claim that
  $\pi_{W_n\cap\,\lspan(a_{n+1})^\perp}(a_n)$ is
  nonzero. Indeed, otherwise $$a_n\perp
  W_n\cap\lspan(a_{n+1})^\perp$$ and
  so $$\pi_{W_n}(a_n)\perp W_n\cap\lspan(a_{n+1})^\perp.$$
  But then, since $a_{n+1}\notin W_n^\perp$ (by
  (\ref{eq:projection})), we have
  that $$W_n\cap(\lspan(\pi_{W_n}(a_n)))^\perp=W_n\cap(\lspan(a_{n+1}))^\perp$$
  and so $a_{n+1}\perp
  W_n\cap(\lspan(\pi_{W_n}(a_n)))^\perp$, which contradicts
  (\ref{eq:projection}). Let then
  $\delta_n=||\pi_{W_n\cap\,\lspan(a_{n+1})^\perp}(a_n)||>0$. This
  ends the construction and the proof.
\end{proof}

Finally, we verify that the stronger property discussed in
Section \ref{sec:trivial} holds for the Hilbert space. Say
that a sequence of tuples $\bar a_n\in\ell_2$ is a
\textit{proper orthogonal sequence} if the subspaces spanned
by different $\bar a_n$ are pairwise ortogonal and the
orthogonal complement of their union is
infinite-dimensional.

\begin{claim}\label{orthogonal}
  Any proper orthogonal sequence in $\ell_2$ is independent.
\end{claim}
\begin{proof}
  Let $\bar a_n$ be a proper orthogonal sequence in the
  quantifier-free type of a given $\bar a$ and let $H_n$ be
  a sequence of orthogonal infinite-dimensional subspaces of
  the orthogonal complement of the space spanned by the
  vectors in all $\bar a_n$'s. Write $H_n'$ for the space
  spanned by $H_n$ and $\bar a_n$ and note that the
  subspaces $H_n'$ witness that the sequence $\bar a_n$ is
  independent.
\end{proof}

\begin{lemma}
  The Hilbert space $\ell_2$ admits independent sequences.
\end{lemma}
\begin{proof}
  Fix $k\in\N$ and $\bar a=(a_1,\ldots,a_k)\in\U^k$. Let
  $(\bar s_n:n\in\N)$ be a sequence of finite tuples and
  without loss of generality assume that $\bar s_n$ is a
  subtuple of $\bar s_{n+1}$. We need to find an independent
  sequence $\bar a_n$ in the quantifier-free type of $\bar
  a$ such that $\bar a_n\equiv_{\bar s_n}\bar a_{n+1}$. Find
  the sequence $\bar a_n$ of tuples as well as additional
  vectors $v_n$ so that
  \begin{itemize}
  \item The elements of $\bar a_n$ are orthogonal to all
    elements of $\bar s_n$, to all elements of $\bar a_i$'s
    for $i<n$ as well as to $v_n$
  \item $v_{n+1}$ is orthogonal to all elements of $\bar
    a_i$'s for $i\leq n$
   \item $\bar a_n\equiv_{\bar s_n}\bar a_{n+1}$
  \end{itemize}
  The sequence is easy to construct using the fact that if
  $\bar b$ and $\bar s$ are two tuples whose elements are
  pairwise orthogonal, then the orbit of $\bar b$ with
  respect to the stabilizer of $\bar s$ contains vectors
  orthogonal to any finite tuple. The sequence is then
  proper orthogonal, and hence independent by Claim
  \ref{orthogonal}.
\end{proof}

\section{Questions}
\label{sec:questions}

There are still many natural examples of automorphism groups
for which the automatic continuity (and even the uniqueness
of Polish group topology) is open. Here we list some of
them.

\begin{question}
  Does the group of automorphisms of the Cuntz algebra
  $\mathcal{O}_2$ have the automatic continuity property?
\end{question}

\begin{question}
  Does the group of automorphisms of the hyperfinite II$_1$
  factor have the automatic continuity property?  
\end{question}

\begin{question}
  Does the group of linear isometries of the Gurari\v\i\ 
  space have the automatic continuity property?
\end{question}

Finally, the problem of uniqueness of separable topology for
the group $\iso(\U)$ remains open. For other groups
considered in this paper, uniqueness of separable topology
follows from the combination of automatic continuity
property and minimality (or even \textit{total minimality}
which says that any Hausdorff quotient of the group is
minimal). For the unitary group this has been proved by
Stojanov \cite{stojanov} and for the group
$\aut([0,1],\lambda)$ by Glasner \cite{glasner} (see also
\cite{byt} for a recent general framework for these kind of
results).

\begin{question}
  Is the group $\iso(\U)$ minimal?
\end{question}

\bigskip
\subsection*{Acknowledgement}
Part of this work has been done during the author's stay at
the University Paris 7 in the academic year
2013$\slash$14. The author is grateful to Zo\'e Chatzidakis,
Amador Mart\'\i n-Pizarro and Todor Tsankov for many useful
comments. The author also wishes to thank Ita\"\i\ Ben
Yaacov, Alexander Kechris and Julien Melleray for inspiring
discussions at the Workshop on Homogeneous Structures during
the special semester at the Hausdorff Institute in
Mathematics in the fall 2013. The author would also like to
thank Piotr Przytycki for valuable discussions and Maciej
Malicki for useful comments.

\bibliographystyle{plain}
\bibliography{refs}

\begin{thebibliography}{10}

\bibitem{ash}
C.~J. Ash.
\newblock Inevitable sequences and a proof of the ``type {${\rm II}$}
  conjecture''.
\newblock In {\em Monash {C}onference on {S}emigroup {T}heory ({M}elbourne,
  1990)}, pages 31--42. World Sci. Publ., River Edge, NJ, 1991.

\bibitem{atim.kallman}
Alexandru~G. Atim and Robert~R. Kallman.
\newblock The infinite unitary and related groups are algebraically determined
  {P}olish groups.
\newblock {\em Topology Appl.}, 159(12):2831--2840, 2012.

\bibitem{becker.kechris}
Howard Becker and Alexander~S. Kechris.
\newblock {\em The descriptive set theory of {P}olish group actions}, volume
  232 of {\em London Mathematical Society Lecture Note Series}.
\newblock Cambridge University Press, Cambridge, 1996.

\bibitem{model.theory}
Ita{\"{\i}} Ben~Yaacov, Alexander Berenstein, C.~Ward Henson, and Alexander
  Usvyatsov.
\newblock Model theory for metric structures.
\newblock In {\em Model theory with applications to algebra and analysis.
  {V}ol. 2}, volume 350 of {\em London Math. Soc. Lecture Note Ser.}, pages
  315--427. Cambridge Univ. Press, Cambridge, 2008.

\bibitem{bybm}
Ita{\"{\i}} Ben~Yaacov, Alexander Berenstein, and Julien Melleray.
\newblock Polish topometric groups.
\newblock {\em Trans. Amer. Math. Soc.}, 365(7):3877--3897, 2013.

\bibitem{byt}
Ita\"\i\ Ben~Yaacov and Todor Tsankov.
\newblock Weakly almost periodic functions, model-theoretic stability, and
  minimality of topological groups.
\newblock 2013.
\newblock preprint.

\bibitem{dales.survey}
H.~G. Dales.
\newblock Automatic continuity: a survey.
\newblock {\em Bull. London Math. Soc.}, 10(2):129--183, 1978.

\bibitem{dales.book}
H.~G. Dales.
\newblock {\em Banach algebras and automatic continuity}, volume~24 of {\em
  London Mathematical Society Monographs. New Series}.
\newblock The Clarendon Press Oxford University Press, New York, 2000.
\newblock Oxford Science Publications.

\bibitem{dudley}
R.~M. Dudley.
\newblock Continuity of homomorphisms.
\newblock {\em Duke Math. J.}, 28:587--594, 1961.

\bibitem{fathi}
A.~Fathi.
\newblock Le groupe des transformations de {$[0, 1]$} qui pr\'eservent la
  mesure de {L}ebesgue est un groupe simple.
\newblock {\em Israel J. Math.}, 29(2-3):302--308, 1978.

\bibitem{fong}
Sourour~A.R. Fong, C.K.
\newblock Normal subgroups of infinite dimensional linear groups.
\newblock 1985.
\newblock DM-365-IR.

\bibitem{fremlin}
D.~H. Fremlin.
\newblock {\em Measure theory. {V}ol. 3}.
\newblock Torres Fremlin, Colchester, 2004.
\newblock Measure algebras, Corrected second printing of the 2002 original.

\bibitem{gao}
Su~Gao.
\newblock {\em Invariant descriptive set theory}, volume 293 of {\em Pure and
  Applied Mathematics (Boca Raton)}.
\newblock CRC Press, Boca Raton, FL, 2009.

\bibitem{gao.kechris}
Su~Gao and Alexander~S. Kechris.
\newblock On the classification of {P}olish metric spaces up to isometry.
\newblock {\em Mem. Amer. Math. Soc.}, 161(766):viii+78, 2003.

\bibitem{glasner}
Eli Glasner.
\newblock The group {${\rm Aut}(\mu)$} is {R}oelcke precompact.
\newblock {\em Canad. Math. Bull.}, 55(2):297--302, 2012.

\bibitem{hall}
Marshall Hall, Jr.
\newblock Coset representations in free groups.
\newblock {\em Trans. Amer. Math. Soc.}, 67:421--432, 1949.

\bibitem{herwig.lascar}
Bernhard Herwig and Daniel Lascar.
\newblock Extending partial automorphisms and the profinite topology on free
  groups.
\newblock {\em Trans. Amer. Math. Soc.}, 352(5):1985--2021, 2000.

\bibitem{hjorth}
Greg Hjorth.
\newblock {\em Classification and orbit equivalence relations}, volume~75 of
  {\em Mathematical Surveys and Monographs}.
\newblock American Mathematical Society, Providence, RI, 2000.

\bibitem{hhls}
Wilfrid Hodges, Ian Hodkinson, Daniel Lascar, and Saharon Shelah.
\newblock The small index property for {$\omega$}-stable {$\omega$}-categorical
  structures and for the random graph.
\newblock {\em J. London Math. Soc. (2)}, 48(2):204--218, 1993.

\bibitem{hrushovski}
Ehud Hrushovski.
\newblock Extending partial isomorphisms of graphs.
\newblock {\em Combinatorica}, 12(4):411--416, 1992.

\bibitem{huhu}
G.~E. Huhunai{\v{s}}vili.
\newblock On a property of {U}ryson's universal metric space.
\newblock {\em Dokl. Akad. Nauk SSSR (N.S.)}, 101:607--610, 1955.

\bibitem{johnson.sinclair}
B.~E. Johnson and A.~M. Sinclair.
\newblock Continuity of derivations and a problem of {K}aplansky.
\newblock {\em Amer. J. Math.}, 90:1067--1073, 1968.

\bibitem{kadison}
Richard~V. Kadison.
\newblock Derivations of operator algebras.
\newblock {\em Ann. of Math. (2)}, 83:280--293, 1966.

\bibitem{kallman.malg}
Robert~R. Kallman.
\newblock Uniqueness results for groups of measure preserving transformations.
\newblock {\em Proc. Amer. Math. Soc.}, 95(1):87--90, 1985.

\bibitem{kallman.homeo}
Robert~R. Kallman.
\newblock Uniqueness results for homeomorphism groups.
\newblock {\em Trans. Amer. Math. Soc.}, 295(1):389--396, 1986.

\bibitem{kaplansky}
Irving Kaplansky.
\newblock Modules over operator algebras.
\newblock {\em Amer. J. Math.}, 75:839--858, 1953.

\bibitem{kechris}
Alexander~S. Kechris.
\newblock {\em Classical descriptive set theory}, volume 156 of {\em Graduate
  Texts in Mathematics}.
\newblock Springer-Verlag, New York, 1995.

\bibitem{kechris.global}
Alexander~S. Kechris.
\newblock {\em Global aspects of ergodic group actions}, volume 160 of {\em
  Mathematical Surveys and Monographs}.
\newblock American Mathematical Society, Providence, RI, 2010.

\bibitem{kechris.miller}
Alexander~S. Kechris and Benjamin~D. Miller.
\newblock {\em Topics in orbit equivalence}, volume 1852 of {\em Lecture Notes
  in Mathematics}.
\newblock Springer-Verlag, Berlin, 2004.

\bibitem{kechris.rosendal}
Alexander~S. Kechris and Christian Rosendal.
\newblock Turbulence, amalgamation, and generic automorphisms of homogeneous
  structures.
\newblock {\em Proc. Lond. Math. Soc. (3)}, 94(2):302--350, 2007.

\bibitem{kr}
Alexander~S. Kechris and Christian Rosendal.
\newblock Turbulence, amalgamation, and generic automorphisms of homogeneous
  structures.
\newblock {\em Proc. Lond. Math. Soc. (3)}, 94(2):302--350, 2007.

\bibitem{kittrell.tsankov}
John Kittrell and Todor Tsankov.
\newblock Topological properties of full groups.
\newblock {\em Ergodic Theory Dynam. Systems}, 30(2):525--545, 2010.

\bibitem{mackey}
George~W. Mackey.
\newblock Ergodic theory and virtual groups.
\newblock {\em Math. Ann.}, 166:187--207, 1966.

\bibitem{malicki}
Maciej Malicki.
\newblock The automorphism group of the {L}ebesgue measure has no non-trivial
  subgroups of index $<2^\omega$.
\newblock {\em Colloq. Math.}, 133(2):169--174, 2013.

\bibitem{mazur.ulam}
Ulam~S. Mazur, S.
\newblock Sur les transformationes isom\'etriques d’espaces vectoriels
  norm\'es.
\newblock {\em C. R. Acad. Sci. Paris}, 194:946--948, 1932.

\bibitem{minasyan}
Ashot Minasyan.
\newblock Separable subsets of {GFERF} negatively curved groups.
\newblock {\em J. Algebra}, 304(2):1090--1100, 2006.

\bibitem{handbook}
J.~Donald Monk and Robert Bonnet, editors.
\newblock {\em Handbook of {B}oolean algebras. {V}ol.\ 3}.
\newblock North-Holland Publishing Co., Amsterdam, 1989.

\bibitem{pestov}
Vladimir Pestov.
\newblock {\em Dynamics of infinite-dimensional groups}, volume~40 of {\em
  University Lecture Series}.
\newblock American Mathematical Society, Providence, RI, 2006.
\newblock The Ramsey-Dvoretzky-Milman phenomenon, Revised edition of {{\i}t
  Dynamics of infinite-dimensional groups and Ramsey-type phenomena} [Inst.
  Mat. Pura. Apl. (IMPA), Rio de Janeiro, 2005; MR2164572].

\bibitem{pin.reutenauer}
Jean-Eric Pin and Christophe Reutenauer.
\newblock A conjecture on the {H}all topology for the free group.
\newblock {\em Bull. London Math. Soc.}, 23(4):356--362, 1991.

\bibitem{ribes.zalesskii}
Luis Ribes and Pavel~A. Zalesskii.
\newblock On the profinite topology on a free group.
\newblock {\em Bull. London Math. Soc.}, 25(1):37--43, 1993.

\bibitem{ringrose}
J.~R. Ringrose.
\newblock Automatic continuity of derivations of operator algebras.
\newblock {\em J. London Math. Soc. (2)}, 5:432--438, 1972.

\bibitem{rosendal.manifolds}
Christian Rosendal.
\newblock Automatic continuity in homeomorphism groups of compact 2-manifolds.
\newblock {\em Israel J. Math.}, 166:349--367, 2008.

\bibitem{rosendal.survey}
Christian Rosendal.
\newblock Automatic continuity of group homomorphisms.
\newblock {\em Bull. Symbolic Logic}, 15(2):184--214, 2009.

\bibitem{rosendal.ribes.zalesskii}
Christian Rosendal.
\newblock Finitely approximable groups and actions part {I}: {T}he
  {R}ibes-{Z}alesski\u\i\ property.
\newblock {\em J. Symbolic Logic}, 76(4):1297--1306, 2011.

\bibitem{rosendal.solecki}
Christian Rosendal and S{\l}awomir Solecki.
\newblock Automatic continuity of homomorphisms and fixed points on metric
  compacta.
\newblock {\em Israel J. Math.}, 162:349--371, 2007.

\bibitem{sakai}
Sh{\^o}ichir{\^o} Sakai.
\newblock On a conjecture of {K}aplansky.
\newblock {\em T\^ohoku Math. J. (2)}, 12:31--33, 1960.

\bibitem{slutsky}
Konstantin Slutsky.
\newblock Automatic continuity for homomorphisms into free products.
\newblock {\em J. Symbolic Logic}.
\newblock to appear.

\bibitem{sol.iso}
S{\l}awomir Solecki.
\newblock Extending partial isometries.
\newblock {\em Israel J. Math.}, 150:315--331, 2005.

\bibitem{steinhaus}
H.~Steinhaus.
\newblock Sur les distances des points dans les ensembles de mesure positive.
\newblock {\em Fund. Math.}, 1:93--104, 1920.

\bibitem{stojanov}
L.~Stojanov.
\newblock Total minimality of the unitary groups.
\newblock {\em Math. Z.}, 187(2):273--283, 1984.

\bibitem{tent.ziegler.new}
Katrin Tent and Martin Ziegler.
\newblock The isometry group of the bounded {U}rysohn space is simple.
\newblock {\em Bull. Lond. Math. Soc.}, 45(5):1026--1030, 2013.

\bibitem{tent.ziegler}
Katrin Tent and Martin Ziegler.
\newblock On the isometry group of the {U}rysohn space.
\newblock {\em J. Lond. Math. Soc. (2)}, 87(1):289--303, 2013.

\bibitem{tsankov}
Todor Tsankov.
\newblock Automatic continuity for the unitary group.
\newblock {\em Proc. Amer. Math. Soc.}, 141(10):3673--3680, 2013.

\bibitem{uspenskij}
Vladimir~V. Uspenskij.
\newblock On subgroups of minimal topological groups.
\newblock {\em Topology Appl.}, 155(14):1580--1606, 2008.

\end{thebibliography}

\end{document}